\documentclass[11pt]{amsart}
\calclayout

\usepackage[utf8]{inputenc}
\usepackage{amsfonts}
\usepackage{amssymb}
\usepackage{color,tikz}
\usepackage{hyperref}
\usepackage[nameinlink]{cleveref}
\usepackage{graphicx}
\usetikzlibrary{calc}
\usetikzlibrary{shapes}

\newcommand{\C}{\mathbb{C}}
\newcommand{\Z}{\mathbb{Z}}
\newcommand{\z}{\mathbf{z}}
\newcommand{\x}{\mathbf{x}}
\newcommand{\R}{R}
\newcommand{\wt}{wt}
\newcommand{\ddelta}{\boldsymbol{\delta}}

\newcommand{\grey}{\mathord{\ast}}
\newcommand{\white}{\mathord{\circ}}
\newcommand{\black}{\mathord{\bullet}}

\numberwithin{equation}{section}

\newtheorem*{theorem*}{Theorem}
\newtheorem*{corollary*}{Corollary}
\newtheorem{thm}{Theorem}[section]

\newtheorem{theorem}[thm]{Theorem}

\newtheorem{corollary}[thm]{Corollary}

\newtheorem{lemma}[thm]{Lemma}

\newtheorem{proposition}[thm]{Proposition}

\newtheorem{prop}[thm]{Proposition}
\newtheorem{claim}[thm]{Claim}

\crefname{theorem}{Theorem}{Theorems}
\crefname{thm}{Theorem}{Theorems}
\crefname{lemma}{Lemma}{Lemmas}
\crefname{lem}{Lemma}{Lemmas}

\theoremstyle{definition}
\newtheorem{defn}[thm]{Definition}
\newtheorem{definition}[thm]{Definition}
\newtheorem{example}[thm]{Example}
\newtheorem{remark}[thm]{Remark}

\newcommand{\pathSW}[2]{

    \coordinate (xy) at #1;
\foreach \di in #2{
     	\ifnum\di=1 
      	\coordinate (next) at ($(xy)+(0,-1)$) ;
	\fi
	\ifnum\di=-1 
      	\coordinate (next) at ($(xy)+(-1,0)$);
	\fi
	\ifnum\di=0 
      	\coordinate (next) at ($(xy)+(-1,-1)$) ;
	\fi
      \draw (xy)--(next);
      \coordinate (xy) at (next);
      }

}

\title{Rhombic staircase tableaux and Koornwinder polynomials}
\date{\today}
\author{Sylvie Corteel}
\address{CNRS, IRIF,  Université Paris Cité, Paris, France and Department of Mathematics, University of California Berkeley, USA.}
\email{corteel@irif.fr}
\author{Olya Mandelshtam}
\address{Department of Combinatorics and Optimization, University of Waterloo, Ontario, Canada}
\email{omandels@uwaterloo.ca}
\author{Lauren Williams}
\address{Department of Mathematics,
Harvard University, Cambridge, MA}
\email{williams@math.harvard.edu}

\begin{document}

\begin{abstract}
In this article we 
give a combinatorial formula
for a certain class of Koornwinder polynomials, 
also known as Macdonald 
polynomials of type 
$\tilde{C}$.
In particular, we give a combinatorial formula
for the Koornwinder polynomials $K_{\lambda} = 
	K_{\lambda}(z_1,\dots,z_N; a,b,c,d; q,t)$,
where $\lambda = (1,\dots,1,0,\dots,0)$.  We also give
combinatorial formulas for 
all ``open boundary ASEP polynomials''
$F_{\mu}$, where $\mu$ is a composition in 
	$\{-1,0,1\}^N$; 
	these polynomials are related to the nonsymmetric Koornwinder polynomials $E_{\mu}$ up to a triangular change of basis. 
	Our formulas are in terms of \emph{rhombic staircase tableaux},
certain tableaux that we introduced in previous work to give 
a formula for the stationary distribution 
 of the 
 two-species asymmetric 
simple exclusion process (ASEP) on a 
line with open boundaries. 
\end{abstract}

\maketitle
        \setcounter{tocdepth}{1}
        \tableofcontents

\section{Introduction}

In recent years there has been a great deal of activity 
connecting the asymmetric simple exclusion process (ASEP), 
special functions, and combinatorics, resulting in new combinatorial 
formulas for various special functions.
For example, Cantini, de Gier, and Wheeler 
\cite{CGW} showed that the specialization of 
the Macdonald polynomial $P_{\lambda}(x_1\dots,x_N; q,t)$
at $x_1= \dots = x_N = 1$ and $q=1$ is the 
partition function for 
the \emph{multispecies ASEP on a ring}.
Subsequently 
 Martin \cite{Martin18}
gave a formula 
for the stationary distribution of the multispecies ASEP on a ring
using some combinatorial objects he called
\emph{multiline queues}.
This led to our discovery of a new formula for Macdonald polynomials
and \emph{ASEP polynomials}, which are related to the nonsymmetric Macdonald polynomials by a triangular change of basis, in terms of multiline queues \cite{CMW18}.

A few years later, Ayyer, Martin, and Mandelshtam \cite{AMM20, AMM22}
proved analogous results, connecting 
the \emph{modified Macdonald polynomials}, 
\emph{queue-inversion tableaux}, 
and the TAZRP (totally asymmetric zero range process). 
In particular, they gave a 
new formula for modified Macdonald polynomials in terms of queue-inversion
tableaux.

\begin{figure}
\begin{center}

\begin{tikzpicture}[scale=0.55]
\def \d {1.5}
\foreach \i\c\t in {0/black/1, 1/gray/2, 2/black/3, 3/gray/4, 4/white/5, 5/white/6, 6/black/7}
{
\node at (\i*\d,-.7) {\tiny \t};}

\node[scale=2] at (1*\d,0) {$\ast$};
\node[scale=2] at (3*\d,0) {$\ast$};
\draw[fill=black] (0*\d,0) circle (0.3cm);
\draw[fill=white] (2*\d,0) circle (0.3cm);
\draw[fill=black] (4*\d,0) circle (0.3cm);
\draw[fill=black] (5*\d,0) circle (0.3cm);
\draw[fill=white] (6*\d,0) circle (0.3cm);

\draw[->] (-1*\d,.5) to [out=60,in=120] (0-.4,.5);
\draw[<-] (-1*\d,-.5) to [out=-60,in=-120] (0-.4,-.5);

\draw[<->] (0+.2,.5) to [out=60,in=120] (1*\d-.2,.5);
\draw[<->] (1*\d+.2,.5) to [out=60,in=120] (2*\d-.2,.5);
\draw[<->] (2*\d+.2,.5) to [out=60,in=120] (3*\d-.2,.5);
\draw[<->] (3*\d+.2,.5) to [out=60,in=120] (4*\d-.2,.5);
\draw[<->] (5*\d+.2,.5) to [out=60,in=120] (6*\d-.2,.5);
\draw[->] (6*\d+.4,.5) to [out=60,in=120] (7*\d,.5);
\draw[<-] (6*\d+.4,-.5) to [out=-60,in=-120] (7*\d,-.5);

\node at (-0.5*\d-.2,1.2) {\tiny${\alpha}$};
\node at (-0.5*\d-.2,-1.2) {\tiny${\gamma}$};

\node at (0.5*\d,1.2) {\tiny${t}$};
\node at (1.5*\d,1.2) {\tiny${t}$};
\node at (2.5*\d,1.2) {\tiny${1}$};
\node at (3.5*\d,1.2) {\tiny${1}$};
\node at (5.5*\d,1.2) {\tiny${t}$};
\node at (6.7*\d,1.2) {\tiny${\beta}$};
\node at (6.7*\d,-1.2) {\tiny${\delta}$};
\end{tikzpicture}\hspace{1cm}
\begin{tikzpicture}[scale=0.5]
\pathSW{(0,0)}{{1,-1,0,1,-1,0,1,-1,1,-1,1,-1}};
\pathSW{(0,0)}{{-1,-1,-1,-1,-1,0,0,1,1,1,1,1}};
\pathSW{(-1,0)}{{1}};
\pathSW{(-1,0)}{{0,1,1}};
\pathSW{(-2,0)}{{0,1,1}};
\pathSW{(-2,0)}{{0,0,1,1,1}};
\pathSW{(-3,0)}{{0,0,1,1,1}};
\pathSW{(-4,0)}{{0,0,1,1,1,1}};
\pathSW{(-2,-1)}{{-1,-1,-1,-1}};
\pathSW{(-2,-2)}{{-1,0,-1,-1,-1}};
\pathSW{(-4,-2)}{{-1,-1,-1}};
\pathSW{(-4,-4)}{{-1,-1,-1}};
\pathSW{(-4,-5)}{{-1,-1,-1}};
\pathSW{(-5,-6)}{{-1,-1}};
\node at (-.5,-.5) {$\alpha $};
\node at (-2.5,-2.5) {$\gamma$};
\node at (-3,-.5) {$\gamma $};
\node at (-4.5,-3.5) {$\beta$};
\node at (-4.5,-4.5) {$\delta$};

\node at (-4,-1.5) {$t$};
\node at (-5,-1.5) {$\alpha t$};
\node at (-4,-.5) {$ t^2$};
\node at (-5.5,-5.5) {$\delta$};
\node at (-6.5,-6.5) {$\gamma$};

\node at (-1.5,-1) {$\beta t$};
\node at (-2.5,-1.5) {$t$};
\node at (-3.5,-2) {$t^2$};
\node at (-3.5,-3) {$t$};

\node at (-4.5,-2.5) {$t$};
\node at (-5.5,-2.5) {$t$};
\node at (-6.5,-2.5) {$t$};
\node at (-5.5,-3.5) {$t$};
\node at (-6.5,-3.5) {$t$};
\end{tikzpicture}	
\end{center}
\caption{On the left: the two-species open ASEP state $\mu=(\black,\grey,\white,\grey,\black,\black,\white)$.  On the right: 
a rhombic staircase tableau of type $\mu$.}
\label{fig:2species}
\end{figure}

In this article we describe a parallel story, this time
relating 
certain
Koornwinder polynomials
 $K_{\lambda} = K_{\lambda}(\z; a,b,c,d; q,t)$,
rhombic staircase tableaux, and 
the open boundary ASEP.
 In \cite{Garbali}, 
Cantini, Garbali, de Gier and Wheeler showed that 
the specialization of the 
Koornwinder polynomial $K_{\lambda}$ 
at $z_1 = \dots = z_N=1$ and $q=1$
is the 
partition function for 
the \emph{multispecies open ASEP}
(the multispecies ASEP on a line with open boundaries), 
generalizing a previous result of Cantini \cite{Can15} for the two-species open ASEP.
In \cite{CMW17}, we gave a formula for the stationary distribution
of the two-species open ASEP in terms of some tableaux we called 
\emph{rhombic staircase tableaux}.
Building on the above insights, in this article we give
 a combinatorial formula 
for the Koornwinder polynomials $K_{\lambda} = 
K_{\lambda}(\z; a,b,c,d; q,t)$,
where $\lambda = (1,\dots,1,0,\dots,0)$, by incorporating 
``spectral parameters'' $z_1^{\pm 1},\dots, z_N^{\pm 1}$ into the rhombic 
staircase tableaux.  We also give
combinatorial formulas for 
the ``open boundary ASEP polynomials'' $f_{\mu}(\z; a,b,c,d; q,t)$ associated to compositions
in $\{-1,0,1\}^N$, which are analogously related to the nonsymmetric Koornwinder polynomials $E_{\mu}(\z; a,b,c,d; q,t)$ by a triangular change of basis.
While our results only treat a special case of Koornwinder polynomials,
this special case is already quite nontrivial; 
as we explain below, Koornwinder polynomials represent a substantial
generalization of Macdonald polynomials.

In the remainder of the introduction we provide some background on  
Koornwinder polynomials, the ASEP, and rhombic staircase tableaux, then
state our main result.

\subsection{Koornwinder polynomials.}

In \cite{Mac88}, Macdonald introduced a family of orthogonal symmetric polynomials in the variables $\x=x_1,x_2,\ldots ,x_N$ and parameters $q, t$ indexed by partitions $\lambda$, associated to different Lie types. In type $A$, these are the classical Macdonald polynomials $P_{\lambda} = P_{\lambda}(\x;q,t)$, which have been studied extensively in representation theory, algebraic combinatorics, and mathematical physics. 
In type $\tilde{C}$, the Macdonald polynomials recover 
the \emph{Koornwinder polynomials} $K_{\lambda} = K_{\lambda}(\z;a,b,c,d;q,t)$
(sometimes called \emph{Macdonald-Koornwinder polynomials}),
 a 6-parameter family of symmetric Laurent polynomials introduced by Koornwinder 
in \cite{Koornwinder}, and 
further studied by van Diejen \cite{Diejen95}, Noumi \cite{NoumiHecke}, 
Sahi \cite{Sahi}, and others. 
We note that the Koornwinder polynomials are a multivariate generalization of the 
the well-known Askey--Wilson polynomials, 
which in turn specialize or limit to all other families of classical hypergeometric orthogonal polynomials in one variable \cite{AW}. 
Moreover the Koornwinder polynomials give 
rise to
the Macdonald polynomials associated to any classical root system
via a limit or specialization; in particular,
the usual (type A) Macdonald polynomial $P_{\lambda}$ is the term
of highest degree $|\lambda|$ in the Koornwinder polynomial 
$K_{\lambda}$
\cite{Diejen95}.

To keep our notation compact, we will often 
write $K_{\lambda}$ or $K_{\lambda}(\z;q,t)$ 
for $K_{\lambda}(\z;a,b,c,d;q,t)$.
When the parts of $\lambda$ are bounded by one, we will write  $K_{\lambda}(\z;t)$, since $K_{\lambda}(\z;q,t)$ is independent of $q$.

There are also nonsymmetric versions of Macdonald and Koornwinder 
polynomials.  
The \emph{nonsymmetric Koornwinder polynomials} $E_{\mu}=
E_{\mu}(\z; a,b,c,d; q,t)$, 
which are indexed by compositions $\mu\in \Z^N$, 
were introduced by Sahi in \cite{Sahi} as the joint eigenfunctions of mutually commuting difference operators constructed from generators of the affine Hecke algebra. For a partition $\lambda$, the (symmetric)
Koornwinder polynomial $K_{\lambda}$ can be obtained as a linear combination of the $E_{\mu}$'s, where $\mu$ ranges over all signed permutations of $\lambda$.

Although there has been a wealth of combinatorics developed to study (symmetric and nonsymmetric) 
Macdonald polynomials in type $A$, such as 
the tableaux formulas of Haglund, Haiman, and Loehr \cite{HHL05}, 
multiline queue formulas \cite{CMW18}, and vertex model formulas \cite{BW22}, 
so far there has been much less progress developing 
the combinatorics of Koornwinder polynomials. 
Ram and Yip gave a type-independent formula for Macdonald polynomials in 
terms of (quantum) alcove paths \cite{RamYip11}, which has been specialized to the Koornwinder case by Orr and Shimozono,
see \cite[Theorem 3.13 and Proposition 3.20]{OS18}.
However, these formulas are much less explicit and combinatorial than 
their type $A$ counterparts.

In this paper we aim to leverage the connection between Koornwinder polynomials
and the open boundary ASEP 
\cite{Can15, CW15, Garbali, FV17}
together with the tableaux formula for the two-species open boundary ASEP \cite{CMW17}
to give tableaux formulas for certain Koornwinder polynomials.

\subsection{The two-species open boundary ASEP}
The \emph{asymmetric simple exclusion process (ASEP)} 
is a canonical example of an out-of-equilibrium system of 
interacting particles on a one-dimensional lattice, originally introduced by Spitzer \cite{Spi70}. 
In the ASEP with open boundaries, 
the particles hop left and right on a one-dimensional finite lattice, 
subject to the constraint that there is at most one particle at each site; particles can also
enter and exit the lattice at the left and right endpoints of the lattice,
at rates $\alpha,\beta,\gamma$ and $\delta$.

There is a multispecies version of the open boundary ASEP, in which 
particles have \emph{weights} $\{0, \pm 1, \ldots , \pm m\}$.  For the purpose of this article, it is 
enough to define the two-species open boundary ASEP, shown at the left of 
\Cref{fig:2species}.

\begin{definition}
The \emph{two-species asymmetric simple exclusion process (ASEP) with open boundaries}
(or \emph{two-species open ASEP})
is a model of interacting particles on a one-dimensional lattice
of $N$ sites, 
such that each site is either vacant (represented by $\white$ or ``-1'') or occupied by a first class particle (represented by $\black$ or ``1'') or a second class particle (represented by $\grey$ or ``0''). 
States of the ASEP are given by words $\mu = (\mu_1,\ldots,\mu_N) \in\{\black,\white,\grey\}^N=\{-1,0,1\}^N$. We say that such a word $\mu$ has \emph{length}
$|\mu| = N$.  First class particles may enter and exit at the left and right boundaries, and both types of particles can hop left or right to adjacent (vacant) sites. 
The transitions in this Markov chain are as follows:

\begin{enumerate}
\item $X\black\white Y\rightarrow X\white\black Y$, $X\black\grey Y\rightarrow X\grey\black Y$, and $X\grey\white Y\rightarrow X\white\grey Y$ with probability $\frac{t}{N+1}$,
\item $X\white\black Y\rightarrow X\black\white Y$, $X\grey\black Y\rightarrow X\black\grey Y$, and $X\white\grey Y\rightarrow X\grey\white Y$ with probability $\frac{1}{N+1}$,
\item $\white Z\rightarrow \black Z$ with probability $\frac{\alpha}{N+1}$,
\item $\black Z\rightarrow \white Z$ with probability $\frac{\gamma}{N+1}$,
\item $Z \black\rightarrow Z \white$ with probability $\frac{\beta}{N+1}$,
\item $Z \white\rightarrow Z \black$ with probability $\frac{\delta}{N+1}$,
\end{enumerate}

where $X,Y,Z \in\{\white,\grey,\black\}^*$ such that $|X|+|Y|=N-2$ and $|Z|=N-1$.
\end{definition}

Note that when there is a $\grey$ at a boundary, no particles may enter or exit at that boundary.
In particular, the number of second class particles is conserved.

Building upon earlier work of \cite{Uchiyama08}, 
the first and third author discovered that the 
 partition 
function $Z_{N,r}$ of the two-species open ASEP on a lattice of 
$N$ sites with $r$ second class particles can be described as 
 a \emph{Koornwinder moment} of type $\lambda=((N-r),0^r)$ \cite{CW15}. 
Concurrently, Cantini 
\cite{Can15} 
found that the partition function $Z_{N,r}$
coincides with the specialization of a Koornwinder polynomial $K_{\lambda}(\z;a,b,c,d;q,t)$ of 
type $\lambda=(1^{N-r},0^r)$ when $z_1= \dots = z_N=1$  after a change of variables according to \cite[(9)]{Can15}.
\footnote{We could also add here
the condition $q=1$, but Koornwinder polynomials $K_{(1,\dots,1,0,\dots,0)}$ in fact have no 
dependence on $q$.} 
In subsequent work, we introduced rhombic staircase tableaux in order to describe
the stationary distribution (and in particular the partition function) of the two-species open ASEP
\cite{CMW17}.

\subsection{Rhombic staircase tableaux}

In \cite{CW11}, Corteel and Williams introduced
\emph{staircase tableaux} and used them to give a combinatorial formula for the stationary distribution
of the open boundary ASEP.  
Subsequently, Mandelshtam and Viennot \cite{MV17}
introduced \emph{rhombic alternative tableaux},
and used them to give a combinatorial formula for the stationary distribution of the 
two-species open ASEP, when $\gamma=\delta=0$.
Finally 
in \cite{CMW17}, we introduced \emph{rhombic staircase tableaux}, and 
used them to give a combinatorial formula for the stationary distribution of the 
two-species open ASEP (with all parameters general).  We now define these tableaux.

\begin{defn}
Let $\mu=(\mu_1,\ldots,\mu_N)\in\{\white, \grey, \black\}^N$. 
We define the \emph{rhombic diagram} $\Gamma(\mu)$ of shape $\mu$ to be the
piecewise linear curve 
whose southeast border is obtained by reading $\mu$ from left to right and 
adding a south step followed by a west step for each $\white$ or $\black$, 
	and a southwest step for each $\grey$. We label the squares and rhombi adjacent to the southeast border from top to bottom with $\{1,\ldots,N\}$. Then a square labelled $j$ corresponds to $\mu_j\in\{\white,\black\}$, and a rhombus labeled $j$ corresponds to $\mu_j=\grey$. The northwest border is obtained by adding $N-r$ west steps followed by $r$ southwest steps followed by $N-r$ south steps to the northeast end of the southeast border, where $r$ are the total numbers of $\grey$ particles in $\mu$. We say such $\Gamma(\mu)$ has \emph{size} $(N,r)$. 
\end{defn}
See \Cref{ex:gamma tiling} for an example.

\begin{defn}
We can tile $\Gamma(\mu)$ using three types of tiles:
\begin{itemize}
\item a \emph{square} tile consisting of horizontal and vertical edges,
\item a \emph{horizontal rhombic} tile (or ``horizontal rhombus'' or ``short rhombus'')
	consisting of horizontal and diagonal edges,
\item and a \emph{vertical rhombic} tile (or ``vertical rhombus'' or ``tall rhombus'')
	consisting of vertical and diagonal edges.
\end{itemize} 
\end{defn}

We will always choose a distinguished tiling of $\Gamma(\mu)$, shown in 
\Cref{ex:gamma tiling}.

\begin{figure}
\begin{center}
\begin{tikzpicture}[scale=0.5]
\pathSW{(0,0)}{{1,-1,0,1,-1,0,1,-1,1,-1,1,-1}};
\pathSW{(0,0)}{{-1,-1,-1,-1,-1,0,0,1,1,1,1,1}};
	\end{tikzpicture}
\hspace{1in}
\begin{tikzpicture}[scale=0.5]
\pathSW{(0,0)}{{1,-1,0,1,-1,0,1,-1,1,-1,1,-1}};
\pathSW{(0,0)}{{-1,-1,-1,-1,-1,0,0,1,1,1,1,1}};
\pathSW{(-1,0)}{{1}};
\pathSW{(-1,0)}{{0,1,1}};
\pathSW{(-2,0)}{{0,1,1}};
\pathSW{(-2,0)}{{0,0,1,1,1}};
\pathSW{(-3,0)}{{0,0,1,1,1}};
\pathSW{(-4,0)}{{0,0,1,1,1,1}};
\pathSW{(-2,-1)}{{-1,-1,-1,-1}};
\pathSW{(-2,-2)}{{-1,0,-1,-1,-1}};
\pathSW{(-4,-2)}{{-1,-1,-1}};
\pathSW{(-4,-4)}{{-1,-1,-1}};
\pathSW{(-4,-5)}{{-1,-1,-1}};
\pathSW{(-5,-6)}{{-1,-1}};
\foreach \i in {1,...,7}
{
\node at (1.2-\i,-.2-\i) {\tiny \i};
}
	\end{tikzpicture}
	\end{center}
\caption{
	$\Gamma(\mu)$ and its distinguished tiling 
when 
	 $\mu=(\black,\grey,\white,\grey,\black,\black,\white)\in\{\white, \grey, \black\}^{7}$, with labels assigned to the tiles on the southeast border.}
\label{ex:gamma tiling}
 \end{figure}

\begin{defn}
A \emph{west-strip} is the maximal contiguous strip consisting of squares and vertical rhombic tiles, where all tiles are adjacent along a vertical edge. Likewise, a \emph{north-strip} is the maximal contiguous strip consisting of squares and horizontal rhombi, where all tiles are adjacent along a horizontal edge. The total number of west- and north-strips in a rhombic diagram of size $(N,r)$ is equal to $N-r$. 
The \emph{distinguished tiling} of a rhombic diagram
is the tiling in 
which all north strips consist of 
squares then horizontal rhombi from bottom to top.
\end{defn}
 See \Cref{ex:gamma tiling} for an example.

\begin{defn}
For $\mu\in\{\white,\grey,\black\}^N$, a \emph{rhombic staircase tableau (RST)} of 
	\emph{type} $\mu$ is a filling of the tiles of $\Gamma(\mu)$ with the letters $\alpha,\beta,\gamma,\delta$ which satisfies the following conditions:
\begin{itemize}
\item A square on the southeast border corresponding to $\mu_j=\white$ must contain a $\beta$ or a $\gamma$.
\item A square on the southeast border corresponding to $\mu_j=\black$ must contain an $\alpha$ or a $\delta$.
\item Each tile is either empty or contains one letter.
\item A horizontal rhombus may contain the letters $\alpha$ or $\gamma$.
\item A vertical rhombus may contain the letters $\beta$ or $\delta$.
\item Every tile in the same north-strip and above $\alpha$ or $\gamma$ must be empty.
\item Every tile in the same west-strip and to the left of $\beta$ or $\delta$ must be empty.
\end{itemize}
\end{defn}

\begin{figure}
\begin{center}
\begin{tikzpicture}[scale=0.5]
\pathSW{(0,0)}{{1,-1,0,1,-1,0,1,-1,1,-1,1,-1}};
\pathSW{(0,0)}{{-1,-1,-1,-1,-1,0,0,1,1,1,1,1}};
\pathSW{(-1,0)}{{1}};
\pathSW{(-1,0)}{{0,1,1}};
\pathSW{(-2,0)}{{0,1,1}};
\pathSW{(-2,0)}{{0,0,1,1,1}};
\pathSW{(-3,0)}{{0,0,1,1,1}};
\pathSW{(-4,0)}{{0,0,1,1,1,1}};
\pathSW{(-2,-1)}{{-1,-1,-1,-1}};
\pathSW{(-2,-2)}{{-1,0,-1,-1,-1}};
\pathSW{(-4,-2)}{{-1,-1,-1}};
\pathSW{(-4,-4)}{{-1,-1,-1}};
\pathSW{(-4,-5)}{{-1,-1,-1}};
\pathSW{(-5,-6)}{{-1,-1}};
\node at (-.5,-.5) {$\alpha $};
\node at (-2.5,-2.5) {$\gamma$};
\node at (-3,-.5) {$\gamma $};
\node at (-4.5,-3.5) {$\beta$};
\node at (-4.5,-4.5) {$\delta$};
\node at (-5,-1.5) {$\alpha$};
\node at (-5.5,-5.5) {$\delta$};
\node at (-6.5,-6.5) {$\gamma$};

\node at (-1.5,-1) {$\beta$};
\end{tikzpicture}
\hspace{1in}
\begin{tikzpicture}[scale=0.5]
\pathSW{(0,0)}{{1,-1,0,1,-1,0,1,-1,1,-1,1,-1}};
\pathSW{(0,0)}{{-1,-1,-1,-1,-1,0,0,1,1,1,1,1}};
\pathSW{(-1,0)}{{1}};
\pathSW{(-1,0)}{{0,1,1}};
\pathSW{(-2,0)}{{0,1,1}};
\pathSW{(-2,0)}{{0,0,1,1,1}};
\pathSW{(-3,0)}{{0,0,1,1,1}};
\pathSW{(-4,0)}{{0,0,1,1,1,1}};
\pathSW{(-2,-1)}{{-1,-1,-1,-1}};
\pathSW{(-2,-2)}{{-1,0,-1,-1,-1}};
\pathSW{(-4,-2)}{{-1,-1,-1}};
\pathSW{(-4,-4)}{{-1,-1,-1}};
\pathSW{(-4,-5)}{{-1,-1,-1}};
\pathSW{(-5,-6)}{{-1,-1}};
\node at (-.5,-.5) {$\alpha $};
\node at (-2.5,-2.5) {$\gamma$};
\node at (-3,-.5) {$\gamma $};
\node at (-4.5,-3.5) {$\beta$};
\node at (-4.5,-4.5) {$\delta$};

\node at (-4,-1.5) {$t$};
\node at (-5,-1.5) {$\alpha t$};
\node at (-4,-.5) {$ t^2$};
\node at (-5.5,-5.5) {$\delta$};
\node at (-6.5,-6.5) {$\gamma$};

\node at (-1.5,-1) {$\beta t$};
\node at (-2.5,-1.5) {$t$};
\node at (-3.5,-2) {$t^2$};
\node at (-3.5,-3) {$t$};

\node at (-4.5,-2.5) {$t$};
\node at (-5.5,-2.5) {$t$};
\node at (-6.5,-2.5) {$t$};
\node at (-5.5,-3.5) {$t$};
\node at (-6.5,-3.5) {$t$};
\end{tikzpicture}	
\end{center}
\caption{On the left, we show a tableau of type $\mu=(\black,\grey,\white,\grey,\black,\black,\white)$, and on the right we show the weights in $t$ associated to each tile. The total weight of the tableau is $\alpha^2\beta^2\delta^2\gamma^3t^{14}$. }	
\label{ex:tableau}
 \end{figure}

\begin{defn}
The \emph{weight} $\wt(T)$ of a rhombic staircase tableau $T$ is a monomial in $\alpha,\beta,\gamma,\delta,t$, obtained by scanning each tile and giving it a weight based on the nearest nonempty tile to its right (\emph{resp.}~below) in its west-strip (\emph{resp.}~north-strip), if those exist:
\begin{itemize}
\item Each horizontal rhombus containing $\alpha$ gets the weight $t$,
\item Each vertical rhombus containing $\beta$ gets the weight $t$,
\item Each empty square that sees $\alpha$ or $\gamma$ to its right and $\alpha$ or $\delta$ below gets the weight $t$
\item Each empty square that sees $\beta$ to its right gets the weight $t$ 
\item Each empty vertical rhombus that sees $\beta$ to its right gets the weight $t^2$
\item Each empty vertical rhombus that sees $\alpha$ or $\gamma$ to its right gets the weight $t$
\item Each empty horizontal rhombus that sees $\alpha$ below gets the weight $t^2$
\item Each empty horizontal rhombus that sees $\beta$ or $\delta$ below gets the weight $t$
\end{itemize}
Finally, $\wt(T)$ is the product of all Greek letters in the filling of $T$ times the product of all weights in $t$ assigned to the tiles according to the rules above. See 
\Cref{ex:tableau} for an example of a tableau of weight $\alpha^2\beta^2\delta^2\gamma^3t^{14}$.
\end{defn}

\begin{definition}\label{def:genpoly}
Given a word $\mu=(\mu_1,\dots,\mu_N) \in
\{-1,0,1\}^N=\{\white,\grey,\black\}^N$, 
let $\R(\mu)=\sum_T \wt(T)$, where the sum is over all 
rhombic staircase tableaux of type $\mu$.  
In particular, $R(\mu)$ is a polynomial in $\alpha,\beta, \gamma, \delta$ and $t$.
\end{definition}

\begin{example} \label{ex:tableaux}
The following figure shows all rhombic staircase tableaux of type $\black\black$:
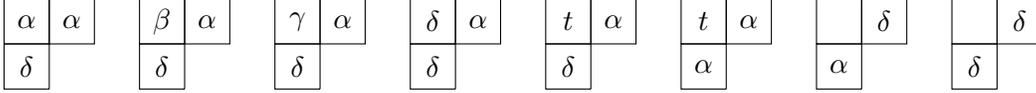
\[
\begin{tikzpicture}[scale=0.6]
\def \w{3};
\foreach \j in {0,...,7}
{
\draw (0+\j*\w,0)--(-1+\j*\w,-0)--(-1+\j*\w,-1)--(0+\j*\w,-1)--(0+\j*\w,0);
\draw (-1+\j*\w,0)--(-2+\j*\w,0)--(-2+\j*\w,-1)--(-1+\j*\w,-1)--(-1+\j*\w,0);
\draw (-1+\j*\w,-2)--(-2+\j*\w,-2)--(-2+\j*\w,-1)--(-1+\j*\w,-1)--(-1+\j*\w,-2);
}
\foreach \j in {0,...,4}
{
\node at (-0.5+\j*\w,-0.5) {$\alpha$};
\node at (-1.5+\j*\w,-1.5) {$\delta$};
}
\node at (-0.5+5*\w,-0.5) {$\alpha$};
\node at (-1.5+7*\w,-1.5) {$\delta$};
\foreach \j in {6,7}
{
\node at (-0.5+\j*\w,-0.5) {$\delta$};
}
\foreach \j in {5,6}
{
\node at (-1.5+\j*\w,-1.5) {$\alpha$};
}

\node at (-1.5,-0.5) {$\alpha$};
\node at (-1.5+\w,-0.5) {$\beta$};
\node at (-1.5+2*\w,-0.5) {$\gamma$};
\node at (-1.5+3*\w,-0.5) {$\delta$};
\node at (-1.5+4*\w,-0.5) {$t$};
\node at (-1.5+5*\w,-0.5) {$t$};

\end{tikzpicture}
\]

	Thus $R(\black \black)=\alpha\delta(1+t+\alpha+\beta+\gamma+\delta)+\alpha^2t+\delta^2.$

To give a second example, the following figure shows all rhombic staircase tableaux of type $\white\grey$:
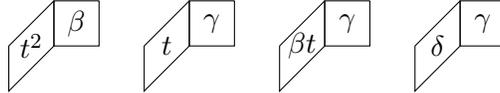
\[
\begin{tikzpicture}[scale=0.6]
\def \w{3};
\foreach \j in {0,1,2,3}
{
\draw (0+\j*\w,0)--(-1+\j*\w,-0)--(-1+\j*\w,-1)--(0+\j*\w,-1)--(0+\j*\w,0);
\draw (-1+\j*\w,0)--(-1+\j*\w,-1)--(-2+\j*\w,-2)--(-2+\j*\w,-1)--(-1+\j*\w,0);
}
\node at (-0.5,-0.5) {$\beta$};
\node at (-0.5+\w,-0.5) {$\gamma$};
\node at (-0.5+2*\w,-0.5) {$\gamma$};
\node at (-0.5+3*\w,-0.5) {$\gamma$};
\node at (-1.5,-1) {$t^2$};
\node at (-1.5+\w,-1) {$t$};
\node at (-1.5+2*\w,-1) {$\beta t$};
\node at (-1.5+3*\w,-1) {$\delta$};
\end{tikzpicture}
\]
	Thus $R(\white \grey)  
	= \beta t^2 + \gamma t + \beta \gamma t + \gamma \delta$.
\end{example}

The main result of \cite{CMW17} was the following.

\begin{theorem}\cite{CMW17} \label{thm:oldCMW}
Consider the two-species open boundary ASEP on a lattice of $N$ sites, with $r$ second class particles.
Let $Z_{N,r} = \sum_\sigma R(\sigma),$ where the sum is over all words $\sigma \in \{-1,0,1\}^N$
containing exactly $r$ $0$'s.
Then $Z_{N,r}$ is the partition function for the two-species open boundary ASEP, 
and for each $\mu \in \{-1,0,1\}^N=\{\white,\grey,\black\}^N$ containing exactly $r$ $0$'s (or $\grey$'s),
 the steady state probability of being in state $\mu$ 
equals $$\frac{R(\mu)}{Z_{N,r}}.$$ 
\end{theorem}

Using \Cref{thm:oldCMW}, it now follows from the result 
of Cantini \cite{Can15} that $Z_{N,r}$ is the specialization of a Koornwinder polynomial 
at $z_1= \dots = z_N$ and $q=1$.
This gives rise to the natural question: how can we incorporate the ``spectral parameters''
$z_1,\dots, z_N$ into the rhombic staircase tableaux, so as to give a 
formula for the Koornwinder polynomials themselves?

\subsection{Main result}
The above question leads us to our main result, described
in \Cref{thm:main}, which is a formula for 
each ``open boundary ASEP polynomial'' $F_{\mu}(\z;t)$, 
 where $\mu$ is a composition in $\{-1,0,1\}^N$. 
 We note that when $\mu=((-1)^{N-r},0^r)$, $F_{\mu}(\z;t)$ 
 coincides with the nonsymmetric Koornwinder polynomial $E_{\mu}(\z;t)$.
 Additionally the sum over all $\mu\in \{-1,0,1\}^N$ with $r$ 0's 
 is the symmetric Koornwinder polynomial $K_{1^{N-r},0^r}(\z;t)$, 
 so we get combinatorial formulas for these polynomials as well.

Recall from \Cref{def:genpoly}
that 
for $\mu=(\mu_1,\dots,\mu_N) \in
\{-1,0,1\}^N=\{\white,\grey,\black\}^N$, 
 $\R(\mu)$ denotes the generating polynomial  in the variables $\alpha,\beta,\gamma,\delta$ and $t$ for the rhombic staircase tableaux of type $\mu$. 
We also let
\begin{equation}
\tilde{\R}(\mu)=\frac{(t-1)^{N-r}}{\prod_{i=2r}^{N+r-1}(\alpha\beta t^i-\gamma\delta)}\R(\mu),
\end{equation} 
where $r$ is the number of zeroes (or $\grey$) in $\mu$.

Let $\mu$ be a word of length $N:=|\mu|$, and let $[N]=\{1,\ldots,N\}$.
 For subsets of indices $I,T\subseteq[N]$ and $u=\mu\vert_I$ a subword of $\mu$ restricted to the indices $I$, define $u\vert_T:=\mu\vert_{T\cap I}$ to be the maximal subword of $u$ supported on $T$. For example, if $\mu=\mu_1\mu_2\mu_3\mu_4\mu_5$, $I=\{2,4,5\}$, $u=\mu_2\mu_4\mu_5$, and $T=\{1,2,4\}$, then $u\vert_T = \mu_2\mu_4$.

We introduce the following 
 change of variables using the parameters $a,b,c,d$\footnote{This change
 of variables is 
nearly the same as \cite[(9)]{Can15} but is different from  \cite[(24), (25)]{Garbali}). }
\begin{align}
\alpha&=\frac{-ac(1-t)}{(a-1)(c-1)},&\qquad \gamma&=\frac{1-t}{(a-1)(c-1)}\nonumber\\
	\beta&=\frac{-bd(1-t)}{(b-1)(d-1)},&\qquad \delta&=\frac{1-t}{(b-1)(d-1)} \label{eq:change}
\end{align}

Let $W_0 = \langle s_1,\dots, s_N\rangle$ be the finite Weyl group of type $C_N$, as in \Cref{def:affine}.
For $S \subseteq[N]$, we let $\overline{S} = 
[N] \setminus S$. 
\begin{theorem}\label{thm:main}
	Let $\lambda \in \{1,0\}^N$ be a partition, 
and let $\ddelta$ be the signed permutation of $\lambda$
such that $\ddelta_1 \leq \ddelta_2 \leq \dots \leq \ddelta_N \leq 0$.
Choose $\mu=(\mu_1,\dots,\mu_N)\in 
W_0 \cdot \lambda \subset \{-1,0,1\}^N=\{\white,\grey,\black\}^N$ and let
$V=\{i \ \vert \ \mu_i \in \{\pm 1\}\}$.

We define 
\begin{equation}\label{eq:phi}
F_{\mu}(\z;t) := 
\sum_{S \subseteq V} \tilde{\R}(\mu\vert_{\overline{S}}) \cdot 
\prod_{i\in S} (z_i^{\mu_i}-1) = 
\sum_{S \subseteq [N]} \tilde{\R}(\mu\vert_{\overline{S}}) \cdot 
\prod_{i\in S} (z_i^{\mu_i}-1).
	\end{equation}

Then as $\mu$ ranges over $W_0 \cdot \lambda$,
the Laurent polynomials 
$$\{F_{\mu}(\z;t) \ \vert \ \mu\in W_0 \cdot \lambda\}$$
form a \emph{qKZ family} (in the sense of \Cref{def:fmu}).

Moreover, if  we use the change of variables from \eqref{eq:change},

\begin{itemize}
	\item $F_{\ddelta}(\z;t)$

equals the nonsymmetric Koornwinder polynomial $E_{\ddelta}(\z;a,b,c,d; t)$,
and 
\item 
	the symmetric Koornwinder polynomial $K_{\lambda}$
is equal to 
	\begin{equation*}
		K_{\lambda}(\z;a,b,c,d;q, t)=
		\sum_{\mu \in W_0 \cdot \lambda}
		F_{\mu}(\z; t),
	\end{equation*}
	where the sum runs over all distinct signed permutations
	of $\lambda$.
\end{itemize}
\end{theorem}
\begin{remark}
We refer to 
the Laurent polynomials $\{F_{\mu}(\z;t)\}$
as \emph{open boundary ASEP polynomials};  each  
$F_{\mu}$ specializes at $z_1=\dots=z_N=1$
to the steady state probability that the two-species
ASEP is in state $\mu$.  
\end{remark}

\begin{remark}\label{rem:leadingcoeff}
It follows from the definition of $F_{\mu}$
that the coefficient of $z^{\mu}$ in $F_{\mu}$ is $1$:
we get this term when $S=V$ because
	$\tilde{R}(\mu)=1$ when $\mu = \grey^r$.
\end{remark}

\begin{remark}
We refer to  the Laurent polynomials $F_{\mu}(z_1,\dots,z_N;t)$
as \emph{open boundary ASEP polynomials} because
they are the open boundary analogue of the \emph{ASEP
	polynomials} (on a ring) which 
were first studied in \cite{CGW, CMW18} (and which were 
subsequently dubbed ``ASEP polynomials'' in \cite{chen-degier-wheeler-2020}).
We note that the Laurent polynomials $F_{\mu}$
are related to the nonsymmetric Koornwinder polynomials
$E_{\mu}$ via a triangular change
of basis, see \Cref{prop:change}.
\end{remark}

\begin{example} For $N=2$, the ASEP polynomials are:
\begin{itemize}
\item $F_{(1,1)}=\widetilde{R}(\black\black)+\widetilde{R}(\black)(z_1+z_2-2)+(z_1-1)(z_2-1)$.
\item $F_{(1,-1)}=\widetilde{R}(\black\white)+\widetilde{R}(\black)(z_2^{-1}-1)+\widetilde{R}(\white)(z_1-1)+(z_1-1)(z_2^{-1}-1)$.
\item $F_{(-1,1)}=\widetilde{R}(\white\black)+\widetilde{R}(\white)(z_2-1)+\widetilde{R}(\black)(z_1^{-1}-1)+(z_1^{-1}-1)(z_2-1)$.
\item $F_{(-1,-1)}=\widetilde{R}(\white\white)+\widetilde{R}(\white)(z_1^{-1}+z_2^{-1}-2)+(z_1^{-1}-1)(z_2^{-1}-1)$.
\item $F_{(1,0)}=\widetilde{R}(\black\grey)+\widetilde{R}(\grey)(z_1-1)$.
\item $F_{(0,1)}=\widetilde{R}(\grey\black)+\widetilde{R}(\grey)(z_2-1)$.
\item $F_{(-1,0)}=\widetilde{R}(\white\grey)+\widetilde{R}(\grey)(z_1^{-1}-1)$.
\item $F_{(0,-1)}=\widetilde{R}(\grey\white)+\widetilde{R}(\grey)(z_2^{-1}-1)$.
\item $F_{(0,0)}=\widetilde{R}(\grey\grey)$.
\end{itemize}
For example, since $\widetilde{R}(\black \black) = \frac{(t-1)^2}{(\alpha\beta t-\gamma\delta)(\alpha\beta -\gamma\delta)}\left(\alpha\delta(1+t+\alpha+\beta+\gamma+\delta)+\alpha^2t+\delta^2\right) $ and $\widetilde{R}(\black)= \frac{t-1}{\alpha\beta -\gamma\delta}(\alpha+\delta)$, we have that 
\begin{multline*}
F_{(1,1)} = \frac{(t-1)^2}{(\alpha\beta t-\gamma\delta)(\alpha\beta -\gamma\delta)}\left(\alpha\delta(1+t+\alpha+\beta+\gamma+\delta)+\alpha^2t+\delta^2\right) \\+ \frac{t-1}{\alpha\beta -\gamma\delta}(\alpha+\delta)(z_1+z_2-2)+(z_1-1)(z_2-1).
\end{multline*}
	We then have that $K_{(1,1)} = F_{(1,1)}+F_{(1,-1)}+F_{(-1,1)}+F_{(-1,-1)}$,
	and $K_{(1,0)} = F_{(1,0)}+F_{(-1,0)} + F_{(0,1)} + F_{(0,-1)}$.
\end{example}

Using \Cref{thm:main}, we 
can also give a simpler combinatorial formula for $K_{\lambda}(\z;q,t)$,
where $\lambda = (1^{N-r},0^r)$.
\begin{theorem}\label{thm:main2}
Let $\lambda=(1^{N-r},0^r)$ and let 
\[
\tilde{Z}_{N,r}=\frac{(t-1)^{N-r}}{\prod_{i=2r}^{N+r-1}(\alpha\beta t^i-\gamma\delta)}Z_{N,r}(\alpha,\beta,\gamma,\delta;t),
\]
 where $Z_{N,r}=Z_{N,r}(\alpha,\beta,\gamma,\delta;t)$ is defined in Theorem \ref{thm:oldCMW} as the generating polynomial for staircase tableaux of size $N$ with $r$ diagonal steps.
Then
	
\begin{equation}
K_{\lambda}(\z;q,t)=\sum_{k=0}^{N-r} \tilde{Z}_{N-k,r}\cdot e_{k}(y_1,\ldots ,y_N),
\end{equation}
where $y_i=z_i+1/z_i-2$ for $1\le i\le N$ and $e_k(y_1,\ldots ,y_N)$ is the  elementary symmetric polynomial.
\label{Ksymm}
\end{theorem}

We now recall that
 Koornwinder polynomials $K_{\lambda}(\z;q,t)$ factor when $q=1$.
This is a Koornwinder analogue of the corresponding factorization
for Macdonald polynomials when $q=1$,
see 
\eqref{eq:factorization}.   
(In \Cref{sec:factor}
we will sketch a proof of \Cref{cor:Koorn at q=1} which we learned from Eric Rains.)

\begin{prop}[{\cite{Rains}}]\label{cor:Koorn at q=1}
Let $\lambda=(\lambda_1,\lambda_2,\ldots,\lambda_N)$ be any partition. 
Then at $q=1$, we have the following formula for the Koornwinder polynomial $K_{\lambda}$:
\[
	K_{\lambda}(\z;1,t)
	=\prod_{i=1}^{\lambda_1} K_{(1^{\lambda_i'},0^{N-\lambda'_i})}(\z;t) 
	=\prod_{i=1}^{\lambda_1} K_{1^{\lambda_i'}}(\z;t),
\]
where $\lambda'$ is the partition conjugate to $\lambda$. 
\end{prop}

\Cref{cor:Koorn at q=1} allows us to 
give a combinatorial formula for 
$K_{\lambda}(\z;1,t)$ for any partition $\lambda$.
\begin{corollary}\label{cor:direct}
The Koornwinder polynomials $K_{\lambda}(\z;q,t)$ at $q=1$ can be written as
\[
K_{\lambda}(\z;1,t)=\sum_{\mu}e_{\mu}(y_1,\ldots ,y_N)\prod_{i=1}^{\lambda_1}\tilde{Z}_{N-\mu_i,N-\lambda'_i},
\]
where $y_i=z_i+1/z_i-2$, and the sum is over compositions $\mu=(\mu_1,\ldots ,\mu_\ell)$ such that $0\le\mu_i\le \lambda'_i$ for $1\le i\le \lambda_1$. 
	
\end{corollary}
We can consider \Cref{cor:direct} as giving a combinatorial
formula for Koornwinder polynomials in terms of sequences of $\ell$ rhombic staircase tableaux.

The structure of this paper is as follows.
In \Cref{sec:Hecke} we define the affine Hecke algebra of type C, Koornwinder polynomials, and the notion of a qKZ family. 
In 
\Cref{sec:operators} we prove our main result by showing that our polynomials $\{F_{\mu}\}$ form a qKZ family. 
And in \Cref{sec:conclusion} we compare this story to its counterpart 
in type A, and discuss further directions and generalizations.

\bigskip

\noindent{\bf Acknowledgements:~} 
We would like to thank the \emph{Research in Paris} 
program at the Institut Henri Poincar\'e 
and the visiting scholars program at Universit\'e Paris Cit\'e 
which made this collaboration possible. 
We are grateful to Luigi Cantini, Eric Rains, Siddharth Sahi, 
and Ole Warnaar 
for useful discussions. We are also grateful to Colin Defant for finding a typo in our manuscript. 
SC was partially supported by NSF grant DMS-2054482, ANR grants ANR-19-CE48-0011 and ANR-18-CE40-0033. OM was partially supported by NSERC grant RGPIN-2021-02568 and NSF grant DMS-1953891. LW was partially supported by the National Science Foundation under Award No.
DMS-2152991. Any opinions, findings, and conclusions or recommendations expressed in this material are
those of the author(s) and do not necessarily reflect the views of the National Science
Foundation.  

\section{The 
Hecke algebra  and Koornwinder polynomials}
\label{sec:Hecke}

In this section we review the affine Hecke algebra of type C, 
 nonsymmetric
and symmetric Koornwinder polynomials, 
the notion of qKZ family, and we explain the relationship between Koornwinder polynomials and qKZ family. For more background on these topics, see 
  \cite{Lusztig89}, 
 \cite{Koornwinder}, 
\cite{Sahi}, and
 \cite{Noumi}.

\subsection{The affine Hecke algebra of type C}

\begin{definition}\label{def:affine}
	The \emph{affine Weyl group} $\mathcal{W}_N$ of type
$\tilde{C}_N$ is the group generated by the elements
$s_0,s_1,\dots,s_N$ subject to the relations
	\begin{align*}
	s_i^2 &=1 \text{ for all }0 \leq i \leq N\\
		s_i s_j &= s_j s_i \text{ if }|i-j|>1\\
		s_i s_{i+1} s_i &= s_{i+1} s_i s_{i+1} \text{ for } 1\leq i \leq N-2\\
		s_0 s_1 s_0 s_1 &=s_1 s_0 s_1 s_0\\
		s_N s_{N-1} s_N s_{N-1} &=s_{N-1} s_N s_{N-1} s_N.
\end{align*}
	The \emph{finite Weyl group of type $C_N$} 
is $W_0 = \langle s_1,\dots, s_N\rangle$.
\end{definition}

Given a parameter $q$, the affine Weyl group $\mathcal{W}_N$ acts on the space
 $\C[z_1^{\pm 1},\ldots,z_N^{\pm 1}]$ of Laurent polynomials in $N$ 
 variables as follows:
\begin{align*}
s_i f(z_1,\ldots,z_N) &= f(z_1,\ldots,z_{i+1},z_i,\ldots,z_N), \qquad 1 \leq i \leq N-1\\
s_0 f(z_1,\ldots,z_N) &= f(qz_1^{-1},z_2,\ldots,z_N)\\
s_N f(z_1,\ldots,z_N) &= f(z_1,\ldots,z_{N-1},z_N^{-1})
\end{align*}

The affine Hecke algebra $\mathcal{H}_N$ of type $\tilde{C}_N$
is a deformation of the group algebra of $\mathcal{W}_N$, which
depends on three parameters $t_0, t_N$ and $t$.  It is generated
by elements $T_0,T_1,\dots,T_N$ subject to the relations
\begin{align}
	\label{eq:w1}
	(T_i-t_i)(T_i+1)&=0
	\text{ for }0 \leq i \leq N\\
	\label{eq:w2}
	T_i T_j &= T_j T_i \text{ if }|i-j|>1\\
	\label{eq:w3}
	T_i T_{i+1} T_i &=T_{i+1} T_i T_{i+1} \text{ for }1 \leq i \leq N-2\\
	\label{eq:w4}
	T_0 T_1 T_0 T_1 &=T_1 T_0 T_1 T_0\\
	\label{eq:w5}
	T_N T_{N-1} T_N T_{N-1} &= T_{N-1} T_N T_{N-1} T_N
\end{align}
where $t_1=t_2=\dots = t_{N-1}=t$.

For $1\leq i \leq N$, define the  operators 

\begin{equation}
	Y_i = (T_i \dots T_{N-1})(T_N \dots T_0)(T_1^{-1} \dots T_{i-1}^{-1}),
\end{equation}
which form an abelian subalgebra, and hence share a common set of eigenfunctions, which are precisely the nonsymmetric Koornwinder polynomials in the polynomial representation of the Hecke algebra $\mathcal{H}_N$.

The following operators, introduced by Noumi \cite{NoumiHecke}, give a polynomial representation 
of $\mathcal{H}_N$.  These operators also appear as the $T_i$'s in \cite[(73)]{Garbali} and are closely related to the $\hat{T}_i$'s in \cite{Can15} (up to swapping $b$ and $c$).

\begin{definition}
We fix parameters $a, b, c, d, t, q$, as well as positive integer $N$.
The following operators act on 
	the space $\C[z_1^{\pm 1},\dots, z_N^{\pm 1}]$ of 
Laurent polynomials in $N$ variables.
\begin{align}
	\widetilde{T}_0&=-\frac{ac}{q} - \frac{(z_1-a)(z_1-c)}{z_1} \cdot \frac{1-s_0}{z_1-qz_1^{-1}} \\
	\widetilde{T}_i &= t - (tz_i-z_{i+1})\cdot \frac{1-s_i}{z_i-z_{i+1}}\\
		& =s_i +\frac{(t-1)(z_is_i -z_{i+1})}{z_i-z_{i+1}}
		\hspace{1cm} \text{ for }1\leq i<N\\
	\widetilde{T}_N &= -bd+ \frac{(bz_N-1)(dz_N-1)}{z_N}\cdot \frac{1-s_N}{z_N-z_N^{-1}}.
\end{align}
\end{definition}

It is straightforward to check that these operators
satisfy the relations \eqref{eq:w1} to \eqref{eq:w5}.

\subsection{Koornwinder polynomials}

We follow the exposition of \cite[Definition 4.1]{Kasatani} and \cite[Definition 1]{Garbali} for the following characterization of nonsymmetric Koornwinder polynomials.

\begin{definition}\label{def:nonsymm} \cite[Definition 1]{Garbali}
Let $\lambda \in \Z^N$ be a composition, and 
let $\lambda^+$ denote the unique dominant element in 
$W_0 \cdot \lambda$, that is, 
$\lambda_1^+ \geq \lambda_2^+ \geq \dots \geq \lambda_N^+ \geq 0$.
Take the shortest element $w\in W_0$ such that 
$w \cdot \lambda^+ = \lambda$, and denote it by 
	$w^+_{\lambda}$.  Let $\rho = (N-1,N-2,\dots, 1, 0)$,
	and $\rho(\lambda) = w_{\lambda}^+ \cdot \rho$.

Then 
	the nonsymmetric Koornwinder polynomial $E_{\lambda}$
	is the unique polynomial which solves the eigenvalue
	equations
	\begin{equation}\label{eq:yi}
		Y_i E_{\lambda} = y_i(\lambda)E_{\lambda} 
		\text{ for }i=1 \dots N, \text{ where }
	\end{equation}
	\begin{equation}\label{t0}
      t_0 = -acq^{-1}, \hspace{1cm} t_N = -bd, \text{ and }
	\end{equation}
	\begin{equation}\label{eq:yi2}
		y_i(\lambda)=q^{\lambda_i} t^{N-i+\rho(\lambda)_i}
		(t_0 t_N)^{\epsilon_i(\lambda)} \text{ and }
		\epsilon_i(\lambda) = 
		\begin{cases}
			1 \text{ if }\lambda_i \geq 0\\
			0 \text{ if }\lambda_i<0
		\end{cases}
	\end{equation}
	and whose coefficient of the term $z^{\lambda}=
	z_1^{\lambda_1} \dots z_N^{\lambda_N}$ is equal to $1$.
\end{definition}

\begin{definition}\label{def:fmu} 
\cite[Definition 3.1]{Kasatani}
Let $\lambda \in \Z^N$ be a composition. Suppose that for each signed permutation $\mu$ of $\lambda$, we have
a Laurent polynomial $f_{\mu} \in \C[z_1^{\pm 1}, \dots,
z_N^{\pm 1}]$ such that the following relations hold:

	\begin{align}
		\widetilde{T}_0 f_{\mu_1,\dots } &= q^{\mu_1} f_{-\mu_1,\dots } 
		\text{ if }\mu_1<0 \label{qkz1}\\
		\widetilde{T}_0 f_{\mu_1,\dots } &= t_0 f_{\mu_1,\dots } 
		\text{ if }\mu_1=0\label{left0}\\
		\widetilde{T}_i f_{\dots,\mu_i,\mu_{i+1},\dots } &= t 
		 f_{\dots , \mu_i, \mu_{i+1},\dots } 
		\text{ if }\mu_i=\mu_{i+1} \label{qkz2}\\
		\widetilde{T}_i f_{\dots,\mu_i,\mu_{i+1},\dots } &= 
		 f_{\dots , \mu_{i+1}, \mu_{i},\dots } 
		\text{ if }\mu_i>\mu_{i+1} \label{qkz3}\\
		\widetilde{T}_N f_{\dots ,\mu_N } &= t_N f_{\dots , \mu_N}
		\text{ if }\mu_N=0\label{right0}\\
		\widetilde{T}_N f_{\dots ,\mu_N } &=  f_{\dots , -\mu_N}
		\text{ if }\mu_N>0 \label{qkz4}
	\end{align}
	Then we call the polynomials $\{f_{\mu}\}$ a
	\emph{qKZ family}.
\end{definition}

The following result appears in \cite{Garbali}, see
\cite[Lemma 3]{Garbali} and the discussion that follows.
We take \Cref{prop:K} as the definition of the symmetric Koornwinder
polynomial.
\begin{proposition}\label{prop:K}
Let $\lambda=(\lambda_1,\dots,\lambda_N)$ be a partition. 
Suppose that 
we have a qKZ family $\{f_{\mu}\}$ as in  \Cref{def:fmu}, 
and suppose further that for $\delta:=-\lambda$, we have
	that $f_{\delta}$ equals the nonsymmetric Koornwinder
	polynomial $E_{\delta}$.  
Then the symmetric Koornwinder polynomial $K_{\lambda}$
is equal to 
	\begin{equation*}
		K_{\lambda}(z_1,\dots,z_N; q,t)=
		\sum_{\mu \in W_0 \cdot \lambda}
		f_{\mu}(z_1,\dots,z_N; q, t),
	\end{equation*}
	where the sum runs over all distinct signed permutations
	of $\lambda$.
\end{proposition}

The following result will be a key ingredient in the proof of our main result.
\begin{proposition}\label{prop:eigen}

Let $\lambda \in \Z^N$ be a composition, and 
let $\ddelta$ be the signed permutation of $\lambda$
such that $\ddelta_1 \leq \ddelta_2 \leq \dots \leq \ddelta_N \leq 0$.
Suppose that we have a qKZ family
	$\{f_{\mu}\}$ as in \Cref{def:fmu}.
	Then 
	$$Y_i f_{\ddelta} = y_i(\ddelta) f_{\ddelta}$$ for $i=1,\dots,N$, i.e.
	\eqref{eq:yi} holds with $f_{\ddelta}$ in place of $E_{\lambda}$.
	Therefore if the coefficient of the term $z^{\ddelta}$ in $f_{\ddelta}$ is equal to $1$,
	then 
	$f_{\ddelta}$ equals 
	the nonsymmetric Koornwinder polynomial
	$E_{\ddelta}$. 
\end{proposition}
\begin{example}
As an example, we compute  $\tilde{\R}(\white\grey)$. From 
\eqref{eq:phi}, we have that $F_{\white\grey}=(z_1^{-1}-1)\tilde{\R}(\grey) + \tilde{\R}(\white\grey) = (z_1^{-1}-1) + \tilde{\R}(\white\grey)$.
From \Cref{prop:eigen}, $Y_1(F_{\white\grey}) = y_i F_{\white\grey} = q^{-1}F_{\white\grey}$, where $Y_1 = T_1T_2T_1T_0$ for $N=2$. We compute 
\begin{align*}
 T_1T_2T_1T_0(z_1^{-1}-1) &= q^{-1}\left((z_1^{-1}-1) -\frac{(t-1)(\beta t^2+\gamma\delta+\beta\gamma t+\gamma t)}{\gamma\delta} \right) \\
  T_1T_2T_1T_0(\tilde{\R}(\white\grey)) &=\frac{\alpha\beta t^2}{\gamma\delta q} \tilde{\R}(\white\grey)
  \end{align*}
  and thus
  \begin{align*}
T_1T_2T_1T_0(F_{\white\grey}) &=q^{-1}\left((z_1^{-1}-1) -\frac{(t-1)(\beta t^2+\gamma\delta+\beta\gamma t+\gamma t)}{\gamma\delta}+ \frac{\alpha\beta t^2}{\gamma\delta} \tilde{\R}(\white\grey)\right)\\
&=q^{-1}\left((z_1^{-1}-1) + \tilde{\R}(\white\grey) \right)
\end{align*}
Solving for $\tilde{\R}(\white\grey)$, we get
\[
\tilde{\R}(\white\grey) = \frac{(t-1)(\beta t^2+\gamma t+\gamma\beta t+\delta\gamma)}{\alpha\beta t^2-\gamma\delta},
\]
	which matches the combinatorial computation for $\R(\white\grey)$ from \Cref{ex:tableaux}. 
\end{example}

\begin{proof}

We start by writing
$\ddelta=(\lambda_1^{k_1},\ldots,\lambda_r^{k_r},0^{N-k})$,
where $\lambda_1<\cdots<\lambda_r<0$ and $(k_1,\ldots,k_r)$ is a (strong) composition of $k$, 
i.e. $\sum k_i = k$, and $k_i>0$ for all $i$.
So $\ddelta^+=(-\lambda_1^{k_1},\ldots,-\lambda_r^{k_r},0^{N-k})$.
The permutation $w_{\ddelta}^+$ is the signed permutation
which divides $\{1,2,\dots,N\}$ into 
consecutive blocks
of sizes $k_1, k_2,\dots, k_r, N-k$, 
and on the first $r$ blocks,
it reverses then negates the elements, and on the last block,
it acts as the identify.
	(E.g. if $\ddelta=(-3,-3,-1,-1,-1,0,0)$,
then $w_{\ddelta}^+$ is the signed permutation
	mapping $(1,2,3,4,5,6,7)$ to $(-2,-1, -5,-4,-3,6,7)$.)
So $\rho(\ddelta)$ is the signed permutation obtained from
$\rho=(N-1,N-2,\dots,1,0)$ by dividing up the domain $(1,2,\dots,N)$ into consecutive blocks of sizes $k_1, k_2,\dots, k_r, N-k$, and on the first $r$ blocks, it reverses then negates the values of $\rho$. In our example, $\rho = (6,5,4,3,2,1,0)$ and $\rho(\ddelta) = (-5,-6,-2,-3,-4,1,0)$.

Thus we have 
\[
\rho(\ddelta)_i=-(N-1-k_1-\cdots-k_j+(i-k_1-\cdots-k_{j-1}))=-N-i+1+2(k_1+\cdots+k_{j-1})+k_j
\]
for $k_1+\cdots+k_{j-1}<i\leq k_1+\cdots+k_j$, and for $k<i\leq N$, 
	\[\rho(\ddelta)_i = N-i.
	\]

For $1\leq i\leq k$, let $j$ be such that $\lambda_j=\ddelta_i$.
\begin{align}
Y_i f_{(\lambda_1^{k_1},\ldots,\lambda_r^{k_r},0^{N-k})}&=(T_i\cdots T_{N-1})(T_N\cdots T_0)(T_1^{-1}\cdots T_{i-1}^{-1})f_{(\lambda_1^{k_1},\ldots,\lambda_j^{k_j},\ldots,\lambda_r^{k_r},0^{N-k})} \nonumber\\
&=(T_i\cdots T_{N-1})(T_N\cdots T_0)t^{-(i-k_1-\cdots-k_{j-1}-1)}f_{(\lambda_j\lambda_1^{k_1},\ldots,\lambda_j^{k_j-1},\ldots,\lambda_r^{k_r},0^{N-k})} \label{eq:h1}\\
&=t^{-(i-k_1-\cdots-k_{j-1}-1)}(T_i\cdots T_{N-1})(T_N\cdots T_1)q^{\ddelta_i}f_{((-\lambda_j)\lambda_1^{k_1},\ldots,\lambda_j^{k_j-1},\ldots,\lambda_r^{k_r},0^{N-k})} \label{eq:h2}\\
&=q^{\ddelta_i}t^{-(i-k_1-\cdots-k_{j-1}-1)}(T_i\cdots T_{N-1})T_N f_{(\lambda_1^{k_1},\ldots,\lambda_j^{k_j-1},\ldots,\lambda_r^{k_r},0^{N-k},(-\lambda_j))} \label{eq:h3}\\
&=q^{\ddelta_i}t^{-(i-k_1-\cdots-k_{j-1}-1)}(T_i\cdots T_{N-1}) f_{(\lambda_1^{k_1},\ldots,\lambda_j^{k_j-1},\ldots,\lambda_r^{k_r},0^{N-k},\lambda_j)} \label{eq:h4}\\
&=q^{\ddelta_i}t^{-(i-k_1-\cdots-k_{j-1}-1)}t^{k_1+\cdots+k_j-i} f_{(\lambda_1^{k_1},\ldots,\lambda_j^{k_j},\ldots,\lambda_r^{k_r},0^{N-k})} \label{eq:h5}\\
&=q^{\ddelta_i}t^{2(k_1+\cdots+k_{j-1})+k_j-2i+1} f_{\ddelta},\nonumber
\end{align}

where \eqref{eq:h1}, \eqref{eq:h3}, and \eqref{eq:h5} are due to \eqref{qkz2}, \eqref{qkz3}, and the fact that $-\ddelta_i>0\geq\ddelta_{\ell}$ for all $\ell$, \eqref{eq:h2} is due to \eqref{qkz1}, and \eqref{eq:h4} is due to \eqref{qkz4}. 

Since $\rho(\ddelta)_i=-N+1-i+2(k_1+\cdots+k_{j-1})+k_j$, the exponent of $t$ is equal to $N-i+\rho(\ddelta)_i$,
so $Y_i f_{\ddelta} = y_i(\ddelta) f_{\ddelta}$ for $1 \leq i \leq k$.

For $k+1\leq i\leq N$, we have
\begin{align}
Y_i f_{(\lambda_1^{k_1},\ldots,\lambda_r^{k_r},0^{N-k})}&=(T_i\cdots T_{N-1})(T_N\cdots T_0)t^{-(i-1-k)}f_{(0,\lambda_1^{k_1},\ldots,\lambda_r^{k_r},0^{N-k-1})} \label{eq:i1}\\
&=t^{k+1-i}(T_i\cdots T_{N-1})(T_N\cdots T_1)t_0f_{(0,\lambda_1^{k_1},\ldots,\lambda_r^{k_r},0^{N-k-1})} \label{eq:i2}\\
&=t_0t^{k+1-i}(T_i\cdots T_{N-1})T_N t^{N-k-1}f_{(\lambda_1^{k_1},\ldots,\lambda_r^{k_r},0^{N-k})} \label{eq:i3}\\
&=t_0t^{N-i}(T_i\cdots T_{N-1})t_N f_{(\lambda_1^{k_1},\ldots,\lambda_r^{k_r},0^{N-k})} \label{eq:i4}\\
&=t_0t_Nt^{N-i}t^{N-i} f_{\ddelta} \label{eq:i5}\\
\end{align}
where \eqref{eq:i1}, \eqref{eq:i2}, and \eqref{eq:i5} are due to \eqref{qkz2} and \eqref{qkz3}, \eqref{eq:i2} is due to \eqref{left0}, and \eqref{eq:i4} is due to \eqref{right0}. 

For $k+1 \leq i \leq N$, we have
$y_i(\ddelta) = t^{N-i+\rho(\ddelta)_i} (t_0 t_N) = t^{2N-2i} t_0 t_N$.

\end{proof}

\subsection{The monomials in the qKZ family}

We define two partial orders on $\Z^N$.  

\begin{definition}
The \emph{dominance order} on compositions in $\Z^N$ is defined as follows:
$\mu \geq \nu$ if for $1 \leq j \leq N$ we have
$\sum_{i=1}^{j} (\mu_i-\nu_i) \geq 0$.
We define a second order $\preceq$ on compositions as follows.
Let $\mu^+$ be the unique element in $W_0 \cdot \mu$ such that
$\mu^+$ is a partition, i.e. such that $\mu_1 \geq \mu_2 \geq \dots \geq 0$. Then we say that 
$\mu \preceq \nu$ if $\mu^+ < \nu^+$, or 
if $\mu^+=\nu^+$ and $\mu \leq \nu$.
\end{definition}

For example, if $\mu=(-2,0)$, then the compositions $\nu\in \Z^2$
such that $\nu \preceq \mu$ are:
\[(-2,0), (1,1), (1,-1), (-1,1),(-1,-1), (1,0), (-1,0),(0,1),(0,-1), (0,0).\]

It is well-known \cite{Sahi} (see also \cite{Can15}) 
that the non-symmetric Koornwinder polynomials
$E_{\mu}(\z)$ have the form 
$$E_{\mu}(\mathbf{z}) = \mathbf{z}^{\mu} + \sum_{\nu \prec \mu} c_{\nu} \mathbf{z}^{\nu}.$$

Moreover, it follows from the definitions
that the nonsymmetric Koornwinder polynomials
$E_{\mu}$ and any qKZ family $f_{\mu}$ are related via a 
triangular change of basis.

\begin{proposition}\cite[Proposition 1]{Garbali}\label{prop:change}
	Let 
	$\{f_{\mu}\}$ be a qKZ family.  Then 
	the nonsymmetric Koornwinder polynomials $E_{\mu}$ and 
	the qKZ family $f_{\mu}$
	are related via an invertible triangular change of basis:
	$$E_{\mu} = \sum_{\nu \preceq \mu} c_{\mu \nu}(q,t)
	f_{\nu}, \text{ and }f_{\mu} = \sum_{\nu \preceq \mu}
	d_{\mu \nu}(q,t) E_{\nu}$$
	for suitable rational coefficients $c_{\mu \nu}(q,t)$
	and $d_{\mu \nu}(q,t)$.
\end{proposition}

It follows that the Laurent monomials appearing in the support 
of each $f_{\mu}$ can be characterized as follows.
\begin{corollary}
Let $\{f_{\mu}\}$ be a qKZ family.
Then $f_{\mu}(\z;q,t)$ has the form
$$f_{\mu}(\z;q,t) = \sum_{\nu \preceq \mu} e_{\mu \nu}(q,t) \z^{\nu}.$$ 
\end{corollary}

\subsection{Factorization of Koornwinder polynomials at $q=1$}\label{sec:factor}

In this section we sketch the proof of the factorization of Koornwinder 
polynomials stated in \Cref{cor:Koorn at q=1}.  We thank Eric Rains for 
explaining this argument to us.

Define $\mathbb{T}^M:=\{(z_1,\ldots,z_M)\in\mathbb{C}^M:|z_1|=\cdots=|z_M|=1\}$ to be the $M$-dimensional complex torus, denote $\z=z_1,\ldots,z_M$, and $dT(\z):=\frac{1}{2^M M!(2\pi i)^M}\frac{dz_1}{z_1}\cdots\frac{dz_M}{z_M}$. The Koornwinder polynomial $K_{\lambda}(\z;q,b,c,d;q,t)$ is characterized as the unique symmetric Laurent polynomial with leading monomial $z^{\lambda}$ that satisfies the orthogonality condition
\[
\int_{\mathbb{T}^M}K_{\lambda}(\z;a,b,c,d;q,t)K_{\mu}(\z;a,b,c,d;q,t)\Delta^M(\z;a,b,c,d;q,t)dT(\z)=0
\]
for $\mu\neq \lambda$, where $\Delta^M(\z;a,b,c,d;q,t)=\Delta^M(z_1,\ldots,z_M;a,b,c,d;q,t)$ is the Koornwinder density. 

At $t=1$, the Koornwinder orthogonality density has the form $\prod_i \Delta(z_i;a,b,c,d;q)$,
where $\Delta$ is the Askey-Wilson density. This leads to the following.

\begin{claim}\label{claim1}
Let $\lambda = (\lambda_1,\dots,\lambda_N)$. The following quantities are proportional:
\begin{align*}
	K_\lambda(z_1,\dots,z_N;a,b,c,d;q,t=1) &\propto 
	\sum_{\pi\in S_N} 
	\prod_{1\le i\le N} p_{\lambda_i}(z_{\pi(i)};a,b,c,d;q),
\end{align*}
	where $p_k(z;a,b,c,d;q)$ is the Askey-Wilson polynomial.
\end{claim}

The constant of proportionality is the size of the stabilizer of the 
composition $\lambda$. 
The following theorem comes from \cite{Mimachi}, 
see also \cite[Theorem 5.18]{Rainspoly}
\begin{theorem}\label{claim2}
The Koornwinder polynomials satisfy a Cauchy identity
\begin{align*}
\sum_{\mu \subseteq M^N}
	& (-1)^{NM-|\mu|}
K_\mu(x_1,\dots,x_N;a,b,c,d;q,t)
K_{N^M-\mu'}(y_1,\dots,y_M;a,b,c,d;t,q) \\
	&=
\prod_{1\le i\le N}\prod_{1\le j\le M}
	(x_i+\frac{1}{x_i}-y_j-\frac{1}{y_j}).
\end{align*}
(Note that $q$ and $t$ are swapped in the second Koornwinder polynomial above.)
\end{theorem}

If we fix a partition $\lambda \subseteq M^N$, with conjugate partition
denoted $\lambda'$, then 
multiply both sides of \cref{claim2} by $K_{N^M-\lambda'}(y_1,\dots,y_M;a,b,c,d;t,q)$, and integrate over $\mathbb{T}^M$ against the Koornwinder density $\Delta^{(M)}(y_1,\dots,y_M;a,b,c,d;t,q)$, we obtain \Cref{cor:1}. 
({For ease of reading, in what follows, we omit the bounds of integration and write $\int{f(\z)}$ to mean $\int_{\mathbb{T}_n}{f(\z)dT(\z)}$.})

\begin{corollary} \label{cor:1}
$K_\lambda  (x_1,\dots,x_N;a,b,c,d;q,t)$ is proportional to
\[        
\int{\prod_{\substack{1\le i\le N\\1\le j\le M}}(x_i+\frac{1}{x_i}-z_j- \frac{1}{z_j})K_{N^M-\lambda'}(z_1,\ldots,z_M;a,b,c,d;t,q)\Delta^{(M)}(\z;a,b,c,d;t,q)}
\]
\end{corollary}

 Taking $q\to 1$ in \Cref{cor:1}, using \Cref{claim1} to expand the Koornwinder polynomial in the integrand, and factoring $\Delta^M$, we obtain that $K_{\lambda}(x_1,\dots,x_N;a,b,c,d;1,t)$ is proportional to

\[        \int{\prod_{\substack{1\le i\le N\\1\le j\le M}}(x_i+1/x_i-z_j-1/z_j)\sum_{\pi \in S_M} \prod_{1\le k\le M} p_{(N^M-\lambda')_k}(z_{\pi(k)};a,b,c,d;t)\prod_{1\le k\le M} \Delta(z_k;a,b,c,d;t)}.
\]

    Now observing that all $M!$ terms in the sum over $\pi\in S_M$ give the same integral tells us that (assuming $\lambda_1\le M$)
\begin{align*}
        K_\lambda & (x_1,\dots,x_N;a,b,c,d;1,t)
\propto \\
        & \int{\prod_{\substack{1\le i\le N\\1\le j\le M}}(x_i+1/x_i-z_j-1/z_j)\left( \prod_{1\le k\le M} p_{(N^M-\lambda')_k}(z_k;a,b,c,d;t)\Delta(z_k;a,b,c,d;t) \right)} \\
        &= 
\prod_{1\le j\le M}
\int{\prod_{1\le i\le N}(x_i+1/x_i-z-1/z) p_{N-\lambda'_{M+1-j}}(z;a,b,c,d;t)\Delta(z;a,b,c,d;t)},
\end{align*}
where the last equality comes from the Fubini theorem.

     The requirement that $\lambda_1\le M$ can be eliminated by rewriting the univariate integral as a coefficient in an expansion in Askey-Wilson polynomials,
  and noting that the coefficient of
  $p_N(z; a, b, c, d;t)$ in $\prod_{1\le j\le N} (z+\frac{1}{z}-x_j-\frac{1}{x_j})$ is $1$.

Let $[p_i]F(z)$ denote the coefficient of $p_i$ in the polynomial $F(z)$.
  We can thus rewrite this as an infinite product:
$$K_\lambda(x_1,\dots,x_N;a,b,c,d;1,t)
\propto
\prod_{1\le i}
[p_{N-\lambda'_i}(z;a,b,c,d;t)]
\prod_{1\le j\le N}(z+\frac{1}{z}-x_j-\frac{1}{x_j})$$

      In particular, we have
$$K_{1^{\ell}}(x_1,\dots,x_N;a,b,c,d;1,t)
\propto
[p_{N-\ell}(z;a,b,c,d;t)]
\prod_{1\le j\le N}(z+\frac{1}{z}-x_j-\frac{1}{x_j}).$$

Thus one can write
$$K_\lambda(x_1,\dots,x_N;a,b,c,d;1,t)
\propto
\prod_{1\le i}
K_{1^{\lambda'_i}}(x_1,\dots,x_N;a,b,c,d;1,t),$$
where the constant can be seen to be $1$ by comparing the leading terms.

\section{The proof of \Cref{thm:main}}
\label{sec:operators}

The main ingredient in our proof of \Cref{thm:main} 
 is the following theorem, which says 
that the Laurent polynomials $\{F_{\mu}(\z;t)\}$ are a qKZ family.
In what follows, if $\mu$ is a word in $\{1,-1,0\}^N$
then we define 
 $||\mu||=N+r$, where $r$ is the number of $0$'s in $\mu$.

\begin{theorem}\label{thm:qKZ}
Let $\lambda \in \{1,-1,0\}^N$ be a composition. 
Then if we let $\mu$ range over all signed permutations of $\lambda$, 
the collection $\{F_{\mu}(\z;t)\}_{\mu} \subset
\C[z_1^{\pm 1}, \dots,
z_N^{\pm 1}]$ 
of Laurent polynomials 
is a qKZ family, in the sense of \Cref{def:fmu}.
\end{theorem}

We will prove \Cref{thm:qKZ} over the course of this section.
Recall that we identify $\black = 1$, $\white=-1$, and 
$\grey = 0$.
We start by recalling some 
``Matrix Ansatz''-style relations among 
the generating polynomials $\R(\mu)$ that previously appeared in
\cite{CMW17}.
 It will turn out 
that the Matrix Ansatz relations are closely related to the action of the Hecke operators. 

\begin{theorem}[{\cite[Theorem 5.4]{CMW17}}]\label{thm:CMW}
For any words $x$ and $y$ in the letters $\{\white,\grey, \black\}=\{-1,0,1\}$,

	\begin{align}
t\R(x\black\white y)&=\R(x\white\black y)+\lambda_{||x||+||y||+2}(\R(x\black y)+\R(x\white y))\\
t\R(x\grey\white y)&=\R(x\white\grey y)+\lambda_{||x||+||y||+3}\R(x\grey y)\\
t\R(x\black\grey y)&=\R(x\grey\black y)+\lambda_{||x||+||y||+3}\R(x\grey y)\\
\beta \R(x\black)&=\delta \R(x\white)+\lambda_{||x||+1}\R(x)\\
\alpha \R(\white x)&=\gamma \R(\black x)+\lambda_{||x||+1}\R(x)
\end{align}
	where $\lambda_{N}=(\alpha\beta t^{N-1}-\gamma\delta)$. 

\end{theorem}

We can rewrite \Cref{thm:CMW} in terms of the $\tilde{\R}(\mu)$, obtaining the following statements.
\begin{corollary}\label{cor:MA}
For any words $x$ and $y$ in the letters $\{\white,\grey, \black\}=\{-1,0,1\}$,
\begin{align}
t\tilde{\R}(x\black\white y)&=\tilde{\R}(x\white\black y)+(t-1)(\tilde{\R}(x\black y)+\tilde{\R}(x\white y))\label{eq:c1a}\\
t\tilde{\R}(x\grey\white y)&=\tilde{\R}(x\white\grey y)+(t-1)\tilde{\R}(x\grey y)\label{eq:c1b}\\
t\tilde{\R}(x\black\grey y)&=\tilde{\R}(x\grey\black y)+(t-1)\tilde{\R}(x\grey y)\label{eq:c1c}\\
\alpha \tilde{\R}(\white x)&=\gamma \tilde{\R}(\black x)+(t-1)\tilde{\R}(x).\label{eq:c1d}\\
\beta \tilde{\R}(x\black)&=\delta \tilde{\R}(x\white)+(t-1)\tilde{\R}(x)\label{eq:c1e}
\end{align}
\end{corollary}

The following lemmas, which are straightforward to prove,
will be helpful for working with the operators.

\begin{lemma}
Let $1 \leq i \leq N-1$, and suppose that $G$ is a polynomial that is independent of $z_i$ and $z_{i+1}$.
Then we have the following.
\begin{align}
	\widetilde{T}_i \left( (z_i-1)\left(\frac{1}{z_{i+1}}-1\right) G\right)
	&=(z_{i+1}-1)\left(\frac{1}{z_{i}}-1\right) G.\label{eq:1}\\
	\widetilde{T}_i \left( (z_i-1) G \right)&=\left((z_{i+1}-1)-(t-1) \right) G.\label{eq:2}\\
		\widetilde{T}_i \left( \left(\frac{1}{z_{i+1}}-1\right) G\right)&=\left( \left(\frac{1}{z_{i}}-1\right)-(t-1)\right) G.\label{eq:3}\\
		\widetilde{T}_i (G) &= tG.\label{eq:4}
	\end{align}
\end{lemma}

\begin{lemma}\label{lem:T0,TN}
Suppose that $G$ is a polynomial that is independent of $z_1$. Then 
\begin{align}
q\widetilde{T}_0\left( (z_1^{-1}-1)G\right) &= (z_1-1)G+\frac{1-t}{\gamma}G. \label{eq:z1}\\
\widetilde{T}_0\left( G\right) &= t_0G. \label{eq:z10}
\end{align}
Suppose that $G$ is a polynomial that is independent of $z_N$. Then 
\begin{align}
\widetilde{T}_N\left( (z_N-1)G\right) &= (z_N^{-1}-1)G+\frac{1-t}{\delta}G.\label{eq:zN}\\
\widetilde{T}_N\left( G\right) &= t_N G \label{eq:zN0}
\end{align}
\end{lemma}

\begin{proof}
If $G$ is independent of $z_1$, then
\[
\frac{1-s_0}{z_1-q z_1^{-1}} G = 0,\qquad \frac{1-s_0}{z_1-q z_1^{-1}}(z_1^{-1}-1)G = \frac{z_1^{-1}-q^{-1}z_1}{z_1-q z_1^{-1}}G = -q^{-1}G.
\]
We also have 
\[
ac=-\frac{\alpha}{\gamma},\quad a+c=-\frac{1-t+\alpha-\gamma}{\gamma}, \quad bd=-\frac{\beta}{\delta}, \quad b+d=-\frac{1-t+\beta-\delta}{\delta}.
\]
Then 
\[
\widetilde{T}_0(G) = -acq^{-1} G = t_0G
\]
and
\begin{align*} 
q\widetilde{T}_0\left((z_1^{-1}-1)G\right)& = -ac(z_1^{-1}-1)G - (z_1-a-c+acz_1^{-1})(-G) \\
&= z_1G +(ac-a-c)G = (z_1-1)+\frac{1-t}{\gamma}G.
\end{align*}
If $G$ is independent of $z_N$, then
\[
\frac{1-s_N}{z_N-z_N^{-1}} G = 0,\qquad  \frac{1-s_N}{z_N-z_N^{-1}}(z_N-1)G = \frac{z_N-z_N^{-1}}{z_N-z_N^{-1}}G = G.
\]
Then 
\[
\widetilde{T}_N(G) = -bd G = t_N G
\]
and
\begin{align*} 
\widetilde{T}_N\left((z_N-1)G\right)& = -bd(z_N-1)G + (bdz_N-b-d+z_N^{-1})G \\
&= z_N^{-1}G +(bd-b-d)G = (z_N^{-1}-1)G+\frac{1-t}{\delta}G.
\end{align*}
\end{proof}

Our first goal is to prove the following.
\begin{theorem}\label{thm:1}
Choose $1 \leq i \leq N-1$.
For any words $x$ and $y$ in the letters $\{\white,\grey, \black\}$, where $|x|=i-1$, we have that 
$$\widetilde{T}_i(F_{x\black \white y}(\z;t)) = F_{x \white \black y}(\z;t).$$
\end{theorem}

\begin{proof}
	Let $\mu=(\mu_1,\dots,\mu_N)=x\black\white y$ and $|x|=i-1$, so that the $\black$ and $\white$ are in positions
$i$ and $i+1$.  
As before, let $V=\{i \ \vert \ \mu_i \in \{\pm 1\}\}$.
	
	In the definition of $F_{\mu}(\z;t)$, we will divide  the sum
	over subsets of $[N]$ into four cases based on whether the subset
	contains both $i$ and $i+1$, just one of them, or neither.  We then get:
\begin{eqnarray*}
	F_{x\black\white y}(\z;t)&=&(z_i-1)\left(\frac{1}{z_{i+1}}-1\right)\sum_{S\subseteq V\setminus \{i,i+1\}}\prod_{j\in S}(z_j^{\mu_j}-1) \tilde{\R}((x\vert_{\overline{S}}) (y\vert_{\overline{S}}))\\
	&&+(z_i-1)\sum_{S\subseteq V\setminus \{i,i+1\}}\prod_{i\in S}(z_i^{\mu_i}-1) \tilde{\R}((x\vert_{\overline{S}}) \white (y\vert_{\overline{S}}))\\
	&&+\left(\frac{1}{z_{i+1}}-1\right)\sum_{S\subseteq V\setminus \{i,i+1\}}\prod_{j\in S}(z_j^{\mu_j}-1) \tilde{\R}((x\vert_{\overline{S}})\black (y\vert_{\overline{S}}))\\
	&& +\sum_{S\subseteq V\setminus \{i,i+1\}}\prod_{j\in S}(z_j^{\mu_j}-1) \tilde{\R}((x\vert_{\overline{S}})\black\white (y\vert_{\overline{S}})).
\end{eqnarray*}

Using \eqref{eq:1}, we get
\begin{eqnarray*}
	&\widetilde{T}_i &\left(  (z_i-1)\left(\frac{1}{z_{i+1}}-1\right)\sum_{S\subseteq V\setminus \{i,i+1\}}\prod_{j\in S}(z_j^{\mu_j}-1) \tilde{\R}((x\vert_{\overline{S}})(y\vert_{\overline{S}})) \right) =\\
&&(z_{i+1}-1)\left(\frac{1}{z_{i}}-1\right)\sum_{S\subseteq V\setminus \{i,i+1\}}\prod_{j\in S}(z_j^{\mu_j}-1) \tilde{\R}((x\vert_{\overline{S}})(y\vert_{\overline{S}})). 
\end{eqnarray*}
Using \eqref{eq:2}, we get
\begin{eqnarray*}
	&&\widetilde{T}_i \left( (z_i-1)\sum_{S\subseteq V\setminus \{i,i+1\}}\prod_{j\in S}(z_j^{\mu_j}-1) \tilde{\R}((x\vert_{\overline{S}})(y\vert_{\overline{S}})\right) =\\
&&(z_{i+1}-1)\sum_{S\subseteq V\setminus \{i,i+1\}}\prod_{j\in S}(z_j^{\mu_j}-1) \tilde{\R}((x\vert_{\overline{S}})\white(y\vert_{\overline{S}}))\\
&&-(t-1) \sum_{S\subseteq V\setminus \{i,i+1\}}\prod_{j\in S}(z_j^{\mu_j}-1) \tilde{\R}((x\vert_{\overline{S}})\white(y\vert_{\overline{S}})).
\end{eqnarray*}

Using \eqref{eq:3}, we get
\begin{eqnarray*}
&&
	\widetilde{T}_i \left( \left(\frac{1}{z_{i+1}}-1\right)\sum_{S\subseteq V\setminus \{i,i+1\}}\prod_{j\in S}(z_j^{\mu_j}-1) \tilde{\R}((x\vert_{\overline{S}})(y\vert_{\overline{S}})) \right) =\\
&&\left(\frac{1}{z_{i}}-1\right)\sum_{S\subseteq V\setminus \{i,i+1\}}\prod_{j\in S}(z_j^{\mu_j}-1) \tilde{\R}((x\vert_{\overline{S}})\black(y\vert_{\overline{S}}))\\
&&-(t-1) \sum_{S\subseteq V\setminus \{i,i+1\}}\prod_{j\in S}(z_j^{\mu_j}-1) \tilde{\R}((x\vert_{\overline{S}}
)\black(y\vert_{\overline{S}})).
\end{eqnarray*}

Using \eqref{eq:4}, we get 
\begin{eqnarray*}
	\widetilde{T}_i \left( \sum_{S\subseteq V\setminus \{i,i+1\}}\prod_{j\in S}(z_j^{\mu_j}-1) \tilde{\R}(\mu\vert_{\overline{S}}) \right)=
t \sum_{S\subseteq V\setminus \{i,i+1\}}\prod_{j\in S}(z_j^{\mu_j}-1) \tilde{\R}(\mu\vert_{\overline{S}}).
\end{eqnarray*}

Therefore 
 \begin{eqnarray*}
\widetilde{T}_i F_{x\black\white y}(\z;t)&=& 
(z_{i+1}-1)\left(\frac{1}{z_{i}}-1\right)\sum_{S\subseteq V\setminus \{i,i+1\}}\prod_{j\in S}(z_j^{\mu_j}-1) \tilde{\R}((x\vert_{\overline{S}})(y\vert_{\overline{S}}))\\ 
&& + (z_{i+1}-1)\sum_{S\subseteq V\setminus \{i,i+1\}}\prod_{j\in S}(z_j^{\mu_j}-1) 
\tilde{\R}((x\vert_{\overline{S}})\white(y\vert_{\overline{S}}))\\
&& +\left(\frac{1}{z_{i}}-1\right)\sum_{S\subseteq V\setminus \{i,i+1\}}\prod_{j\in S}(z_j^{\mu_j}-1) \tilde{\R}((x\vert_{\overline{S}})\black(y\vert_{\overline{S}}))\\
&&+ \sum_{S\subseteq V\setminus \{i,i+1\}}\prod_{j\in S}(z_j^{\mu_j}-1) (t\tilde{\R}((x\vert_{\overline{S}})\black\white(y\vert_{\overline{S}}))\\
&&-(t-1)
	 (\tilde{\R}((x\vert_{\overline{S}})\black(y\vert_{\overline{S}}))+\tilde{\R}((x\vert_{\overline{S}})\white(y\vert_{\overline{S}}))).
\end{eqnarray*}

Now we use \eqref{eq:c1a}  
 and get that
\[
t\tilde{\R}((x\vert_{\overline{S}})\black\white(y\vert_{\overline{S}}))-(t-1)
(\tilde{\R}((x\vert_{\overline{S}})\black(y\vert_{\overline{S}}))+\tilde{\R}((x\vert_{\overline{S}})\white(y\vert_{\overline{S}})))=\tilde{\R}((x\vert_{\overline{S}})\white\black(y\vert_{\overline{S}})).
\]

So
\begin{eqnarray*}
\widetilde{T}_i F_{x\black\white y}(\z;t)&=& 
	\left(\frac{1}{z_{i}}-1\right)
(z_{i+1}-1)
	\sum_{S\subseteq V\setminus \{i,i+1\}}\prod_{j\in S}(z_j^{\mu_j}-1) \tilde{\R}((x\vert_{\overline{S}})(y\vert_{\overline{S}}))\\ 
&& + (z_{i+1}-1)\sum_{S\subseteq V\setminus \{i,i+1\}}\prod_{j\in S}(z_j^{\mu_j}-1) \tilde{\R}((x\vert_{\overline{S}})\white(y\vert_{\overline{S}}))\\
&& +\left(\frac{1}{z_{i}}-1\right)\sum_{S\subseteq V\setminus \{i,i+1\}}\prod_{j\in S}(z_j^{\mu_j}-1) \tilde{\R}((x\vert_{\overline{S}})\black(y\vert_{\overline{S}}))\\
&&+ \sum_{S\subseteq V\setminus \{i,i+1\}}\prod_{j\in S}(z_j^{\mu_j}-1) \tilde{\R}((x\vert_{\overline{S}})\white\black(y\vert_{\overline{S}}))\\
&=&  F_{x\white\black y}(\z;t),
\end{eqnarray*}
as desired.
\end{proof}

We also need to prove the following identities.

\begin{theorem}\label{thm:2}
Choose $1 \leq i \leq N-1$.
For any words $x$ and $y$ in the letters $\{\white,\grey, \black\}$, 
	where $|x|=i-1$, we have that 
\begin{equation}\label{xcirc}
\widetilde{T}_i(F_{x\grey \white y}(\z;t)) = F_{x \white \grey y}(\z;t)
\end{equation}
and
\begin{equation}\label{bulletx}
\widetilde{T}_i(F_{x\black \grey y}(\z;t)) = F_{x \grey \black y}(\z;t).
\end{equation}
\end{theorem}

\begin{proof}
We prove \eqref{xcirc}. Let $\mu=(\mu_1,\dots,\mu_N)=x\grey\white y$ and $|x|=i-1$, so that the $\grey$ and $\white$ are in positions
$i$ and $i+1$. As before, let $V=\{i \ \vert \ \mu_i \in \{\pm 1\}\}$.
	
We divide the sum in the definition of $F_{\mu}(\z;t)$
	over subsets of $[N]$ into two cases based on whether the subset
	contains $i+1$ or not.  We then get:
\begin{eqnarray*}
	F_{x\grey\white y}(\z;t)&=&\left(z_{i+1}^{-1}-1\right)\sum_{S\subseteq V\setminus \{i+1\}}\prod_{j\in S}(z_j^{\mu_j}-1) \tilde{\R}((x\vert_{\overline{S}})\grey(y\vert_{\overline{S}}))\\
	&& +\sum_{S\subseteq V\setminus \{i+1\}}\prod_{j\in S}(z_j^{\mu_j}-1) \tilde{\R}((x\vert_{\overline{S}})\grey\white(y\vert_{\overline{S}})).
\end{eqnarray*}

Applying \eqref{eq:3} and \eqref{eq:4} to the two sums in $\widetilde{T}_i(F_{\mu}(\z;t))$, we get
\begin{align*}
\widetilde{T}_i(F_{x\grey\white y}(\z;t))=&\left(z_{i}^{-1}-1\right)\sum_{S\subseteq V\setminus \{i+1\}}\prod_{j\in S}(z_j^{\mu_j}-1) \tilde{\R}((x\vert_{\overline{S}})\grey(y\vert_{\overline{S}}))\\
&-(t-1) \sum_{S\subseteq V\setminus \{i+1\}}\prod_{j\in S}(z_j^{\mu_j}-1) \tilde{\R}((x\vert_{\overline{S}})\grey(y\vert_{\overline{S}}))\\
&+t \sum_{S\subseteq V\setminus \{i+1\}}\prod_{j\in S}(z_j^{\mu_j}-1) \tilde{\R}((x\vert_{\overline{S}})\grey\white(y\vert_{\overline{S}}))\\
=&\left(z_{i}^{-1}-1\right)\sum_{S\subseteq V\setminus \{i+1\}}\prod_{j\in S}(z_j^{\mu_j}-1) \tilde{\R}((x\vert_{\overline{S}})\grey(y\vert_{\overline{S}}))\\
\quad+\sum_{S\subseteq V\setminus \{i+1\}}&\prod_{j\in S}(z_j^{\mu_j}-1) ( t \tilde{\R}((x\vert_{\overline{S}})\grey\white(y\vert_{\overline{S}})) -(t-1)\tilde{\R}((x\vert_{\overline{S}})\grey(y\vert_{\overline{S}}))).
\end{align*}

Using \eqref{eq:c1b}, we get  
\begin{eqnarray*}
\widetilde{T}_i(F_{x\grey\white y}(\z;t))=
&&\left(z_{i}^{-1}-1\right)\sum_{S\subseteq V\setminus \{i+1\}}\prod_{j\in S}(z_j^{\mu_j}-1) \tilde{\R}((x\vert_{\overline{S}})\grey(y\vert_{\overline{S}}))\\
&&+\sum_{S\subseteq V\setminus \{i+1\}}\prod_{j\in S}(z_j^{\mu_j}-1) \tilde{\R}((x\vert_{\overline{S}})\white\grey(y\vert_{\overline{S}})),
\end{eqnarray*}
which is equal to $F_{x\white\grey y}(\z;t)$ when written as a sum of the two cases depending on whether or not $i$ is in $S$.

\eqref{bulletx} is proved in the same manner.
\end{proof}

\begin{theorem}\label{thm:4}
For any word $x$ in the letters $\{\white,\grey, \black\}$ of length $|x|=N-1$, we have that 
\[
q\widetilde{T}_0(F_{\white x}(\z;t)) = F_{\black x}(\z;t).
\] 
\end{theorem}

\begin{proof}
Let $\mu=(\mu_1,\dots,\mu_N)=\white x$ and $|x|=N-1$, so that $\white$ is in position
$1$. As before, let $V=\{i \ \vert \ \mu_i \in \{\pm 1\}\}$.
	
We divide the sum over subsets of $[N]$ in the definition of $F_{\mu}(\z;t)$ into two cases based on whether the subset contains $1$ or not.  We get:
\begin{eqnarray*}
	F_{\white x}(\z;t)&=&\left(z_1^{-1}-1\right)\sum_{S\subseteq V\setminus \{1\}}\prod_{j\in S}(z_j^{\mu_j}-1) \tilde{\R}(x\vert_{\overline{S}})\\
	&& +\sum_{S\subseteq V\setminus \{1\}}\prod_{j\in S}(z_j^{\mu_j}-1) \tilde{\R}(\white(x\vert_{\overline{S}})).
\end{eqnarray*}
Thus
\begin{eqnarray*}
	q\widetilde{T}_0(F_{\white x}(\z;t))&=&\left(z_1-1\right)\sum_{S\subseteq V\setminus \{1\}}\prod_{j\in S}(z_j^{\mu_j}-1) \tilde{\R}(x\vert_{\overline{S}})\\
	&&+\frac{1-t}{\gamma}\sum_{S\subseteq V\setminus \{1\}}\prod_{j\in S}(z_j^{\mu_j}-1) \tilde{\R}(x\vert_{\overline{S}})\\
	&& +\frac{\alpha}{\gamma} \sum_{S\subseteq V\setminus \{1\}}\prod_{j\in S}(z_j^{\mu_j}-1) \tilde{\R}(\white(x\vert_{\overline{S}})).
\end{eqnarray*}

Now we use \eqref{eq:c1d} to get that 
\[
\tilde{\R}(\black (x\vert \overline{S})) = \frac{1}{\gamma}\left( \alpha \tilde{\R}(\white(x\vert_{\overline{S}})) +(1-t)\tilde{\R}(x\vert_{\overline{S}})\right).
\]
Then 
\begin{eqnarray*}
	q\widetilde{T}_0(F_{\white x}(\z;t))&=&\left(z_1-1\right)\sum_{S\subseteq V\setminus \{1\}}\prod_{j\in S}(z_j^{\mu_j}-1) \tilde{\R}(x\vert_{\overline{S}})\\
	&&+\sum_{S\subseteq V\setminus \{1\}}\prod_{j\in S}(z_j^{\mu_j}-1) \tilde{\R}(x\vert_{\overline{S}})\left(\frac{1-t}{\gamma}+\frac{\alpha}{\gamma}\right)\\
	&=&F_{\black x}(\z;t).
\end{eqnarray*}
\end{proof}

\begin{theorem}\label{thm:5}
For any word $x$ in the letters $\{\white,\grey, \black\}$ of length $|x|=N-1$, we have that 
\[
\widetilde{T}_N(F_{x \black}(\z;t)) = F_{x \white}(\z;t).
\] 
\end{theorem}

\begin{proof}
Let $\mu=(\mu_1,\dots,\mu_N)=x\black$ and $|x|=N-1$, so that $\black$ is in position
$N$. As before, let $V=\{i \ \vert \ \mu_i \in \{\pm 1\}\}$.
	
We divide the sum over subsets of $[N]$ in the definition of $F_{\mu}(\z;t)$ into two cases based on whether the subset contains $N$ or not.  We get:
\begin{eqnarray*}
	F_{x\black}(\z;t)&=&\left(z_N-1\right)\sum_{S\subseteq V\setminus \{N\}}\prod_{j\in S}(z_j^{\mu_j}-1) \tilde{\R}(x\vert_{\overline{S}})\\
	&& +\sum_{S\subseteq V\setminus \{N\}}\prod_{j\in S}(z_j^{\mu_j}-1) \tilde{\R}((x\vert_{\overline{S}})\black).
\end{eqnarray*}
Thus
\begin{eqnarray*}
	\widetilde{T}_N(F_{x\black}(\z;t))&=&\left(z_N^{-1}-1\right)\sum_{S\subseteq V\setminus \{N\}}\prod_{j\in S}(z_j^{\mu_j}-1) \tilde{\R}(x\vert_{\overline{S}})\\
	&&+\frac{1-t}{\delta}\sum_{S\subseteq V\setminus \{N\}}\prod_{j\in S}(z_j^{\mu_j}-1) \tilde{\R}(x\vert_{\overline{S}})\\
	&& +\frac{\beta}{\delta} \sum_{S\subseteq V\setminus \{N\}}\prod_{j\in S}(z_j^{\mu_j}-1) \tilde{\R}(x\vert_{\overline{S}}).
\end{eqnarray*}

Now we use \eqref{eq:c1e} to get that 
\[
\tilde{\R}(x\vert_{\overline{S}}) = \frac{1}{\delta}\left( \beta \tilde{\R}((x\vert_{\overline{S}})\black) +(1-t)\tilde{\R}(x\vert_{\overline{S}})\right).
\]
Then 
\begin{eqnarray*}
	\widetilde{T}_N(F_{x\black}(\z;t))&=&\left(z_N^{-1}-1\right)\sum_{S\subseteq V\setminus \{N\}}\prod_{j\in S}(z_j^{\mu_j}-1) \tilde{\R}(x\vert_{\overline{S}})\\
	&&+\sum_{S\subseteq V\setminus \{N\}}\prod_{j\in S}(z_j^{\mu_j}-1) \tilde{\R}(x\vert_{\overline{S}})\left(\frac{1-t}{\delta}+\frac{\beta}{\delta}\right)\\
	&=&F_{x\white}(\z;t).
\end{eqnarray*}
\end{proof}

We also prove the following identity. 

\begin{thm}\label{thm:6}
Choose words $x$ and $y$ in the letters $\{\white,\grey,\black\}$ where $|x|=i-1$ and 
additionally choose  $a\in\{\white,\grey,\black\}$.  Then

\begin{equation}\label{eq:aa}
\widetilde{T}_i(F_{x a a y}(\z;t)) = tF_{x a a y}(\z;t).
\end{equation}
\end{thm}

\begin{proof} Let $\mu=(\mu_1,\ldots ,\mu_N)=x a a y$.
If $a=\grey$, the equality trivially holds by \eqref{eq:4}. Otherwise, we will show $F_{x a a y}(\z;t)$ is symmetric in variables $z_i,z_{i+1}$. 
Let $u=\mu_i=\mu_{i+1}$. Then
\begin{eqnarray*}
	F_{xaay}(\z;t)&=&(z_i^u-1)(z_{i+1}^u-1)\sum_{S\subseteq V\setminus \{i,i+1\}}\prod_{j\in S}(z_j^{\mu_j}-1) \tilde{\R}((x\vert_{\overline{S}})(y\vert_{\overline{S}}))\\
	&&+((z_i^u-1)+(z_{i+1}^u-1))\sum_{S\subseteq V\setminus \{i,i+1\}}\prod_{j\in S}(z_j^{\mu_j}-1) \tilde{\R}((x\vert_{\overline{S}})a(y\vert_{\overline{S}}))\\
	&& +\sum_{S\subseteq V\setminus \{i,i+1\}}\prod_{j\in S}(z_j^{\mu_j}-1) \tilde{\R}((x\vert_{\overline{S}})aa(y\vert_{\overline{S}})).
\end{eqnarray*}

Thus $F_{xaay}(\z;t)$ is symmetric in $z_i,z_{i+1}$, so $\frac{1-s_i}{z_i-z_{i+1}}F_{xaay}(\z;t)=0$, from which \eqref{eq:aa} follows.
\end{proof}

Finally we are ready to put together
our previous results to prove \Cref{thm:qKZ}.
\begin{proof}
First note that \Cref{thm:4} proves 
\eqref{qkz1}.
It follows from \eqref{eq:phi} that 
for any $\mu$ with $\mu_i=\grey$,
$\tilde{\R}(\mu)$ is independent of $z_i$.
Hence 
 \eqref{eq:z10} 
implies \eqref{left0}
 and \eqref{eq:zN0}  implies
 \eqref{right0}.
Finally we have that \Cref{thm:6} proves \eqref{qkz2}, 
\Cref{thm:1,thm:2} proves \eqref{qkz3}, and 
\Cref{thm:5} proves 
\eqref{qkz4}. 
\end{proof}

Now we can put all these ingredients together to complete the proof of the main theorem.
\begin{proof}[Proof of \Cref{thm:main}]
	By \Cref{thm:qKZ}, the $\{F_{\mu}(\z ; t)\}_{\mu}$ are a qKZ family
	and that $Y_i F_{\ddelta} = y_i(\ddelta) F_{\ddelta}$ for $i=1,\dots,N$.  Together with \Cref{def:nonsymm} 
	and the fact that the leading coefficient of $F_{\ddelta}$ is $\z^{\ddelta}$ (see  \Cref{rem:leadingcoeff}), this implies that $F_{\ddelta}$ equals the 
	nonsymmetric Koornwinder polynomial $E_{\ddelta}$. The theorem now follows from \Cref{prop:K}.
\end{proof}

\begin{proof}[Proof of 
\Cref{Ksymm}]
The proof of Theorem \ref{Ksymm} follows from the \Cref{thm:main}. 

\begin{eqnarray*}
	K_{\lambda}(\z;t)&=&\sum_{\mu\in W_0(\lambda)} F_{\mu}(\z;t)\\
&=&\sum_{\mu\in W_0(\lambda)}\ \sum_{S \subseteq [N]}
\tilde{R}(\mu|_{\overline{S}}) \prod_{i\in S}(z_i^{\mu_i}-1)\\
	&=& \sum_{k=0}^{N-r}\ \sum_{S \in {[N] \choose k}} \ \ 
	\sum_{\nu \in W_0(1^{N-k},0^r)} 
	\tilde{R}(\nu) \cdot 
	\prod_{i\in S} \left( (z_i-1) + (z_i^{-1}-1) \right)\\
	&=&\sum_{k=0}^{N-r} \tilde{Z}_{N-k,r}\cdot e_k(z_1+z_1^{-1}-2,\ldots ,z_N+z_N^{-1}-2).
\end{eqnarray*}
where $W_0$ is the set of signed permutations of the appropriate cardinality,
so e.g. 
the sum over $\nu \in W_0(1^{N-k},0^r)$ indicates 
that we are summing over all signed permutations of the 
word $(1^{N-k},0^r)$.

\end{proof}

\section{Conclusion}\label{sec:conclusion}

In this work we have given a combinatorial formula 
for the Koornwinder polynomials $K_{\lambda}$ associated
to partitions $\lambda=(1^{N-r},0^r)$ by using the
combinatorics 
of the open boundary two-species ASEP.  
This work naturally leads to several
open problems which we hope to pursue in future work.

A first problem
is to find a vertex model analogue of our rhombic staircase tableaux,
corresponding to particle models with open boundaries and their associated type $\tilde{C}$ special functions. 
We note that there has been a great deal of recent work 
connecting particle models on the ring and (type $A$) special functions to 
\emph{vertex models}, see \cite{BW22, ABW21, ANP23}.
In particular, \cite[Section 6]{CGW} gave a lattice model interpretation of their 
matrix product formula for Macdonald polynomials.  Subsequently, 
\cite{ANP23} gave a queue vertex model formula for Macdonald polynomials,
which encodes the multiline queues of Martin \cite{Martin18}, and gives new proofs of 
some formulas for the stationary distribution of the multispecies ASEP on a ring
using the Yang--Baxter equation.  

A second problem is to extend our results to give a combinatorial formula for all Koornwinder polynomials. Towards that end, we now briefly sketch how the analogous problem was solved for type $A$ Macdonald polynomials,
and what are the difficulties in the Koornwinder case. 

Let us refer to the multispecies ASEP on a ring involving particles 
$\{0,1,\dots,m\}$ as the \emph{rank $m$} multispecies ASEP on a ring.
If $m=1$, the multispecies ASEP on a ring is rather trivial --
its stationary distribution is uniform -- but for $m>1$, the stationary distribution
becomes more interesting.  To get some insight on what happens for $m>1$, recall that 
the multispecies ASEP on a ring is connected to 
the type $A$ Macdonald polynomials.  More specifically,
the partition function of the rank $m$ multispecies ASEP on a ring
is the specialization of a related Macdonald polynomial (with largest part $m$)
at $x_1=\cdots=x_N=1$ and $q=1$ \cite{CGW}.  Moreover, when $q=1$, 
Macdonald polynomials admit the following factorization \cite[Chapter VI, Equation (4.14)(vi)]{Macdonald}:
\begin{equation}\label{eq:factorization}
	P_{\lambda}(\x;1,t)=
	\prod_{i\ge 1}P_{(1^{\lambda'_i}, 0^{N-\lambda'_i})}(\x;1,t)=
	\prod_{i\ge 1}P_{(1^{\lambda'_i})}(\x;1,t)=
	\prod_{i\ge 1}e_{{\lambda'_i}}(\x),
\end{equation}

where $e_{\mu}$ is the elementary symmetric polynomial.
This means that the partition function of the rank $m$ mASEP on a ring is just the product
of some elementary symmetric polynomials.  Martin's 
multiline queue formula for the stationary distribution of the mASEP \cite{Martin18}
reflects this structure: each multiline queue is built by stacking 
rows of balls on top of each other, where each row of balls can be thought of as describing
a rank $1$ ASEP.  
Once one 
has the multiline queues with the $t$ statistic, 
adding the parameters $q$ and $x_i$ to 
obtain a formula for Macdonald polynomials 
turns out to be miraculously simple \cite{CMW18}.

Returning to the ASEP with open boundaries, let us
refer to the multispecies ASEP with 
particles $\{0,\pm1,\pm2,\ldots,\pm m\}$ as the \emph{rank $m$} open boundary ASEP. 
Unlike the ASEP on a ring, the rank $1$ open boundary ASEP has a very non-trivial stationary distribution.
Even in the case
of the rank $1$ open ASEP with no second class particles, 
there was a tremendous amount of work on the ASEP 
 \cite{DEHP93, USW04, BE, 
jumping, BCEPR,  CW1} before
a combinatorial formula for the stationary distribution (with all parameters $\alpha, \beta, \gamma, \delta, q$ 
general) was given in \cite{CW11};
this formula was subsequently extended to the rank $1$ case with second class particles
in \cite{CMW17}, using rhombic staircase tableaux.
Similar to the case of the ASEP on a ring, 
the partition function of the rank $m$ open boundary mASEP 
is the specialization of a related Koornwinder polynomial (with largest part $m$)
at $x_1=\cdots=x_N=1$ and $q=1$ \cite{Garbali}.  Moreover, just as Macdonald polynomials factor at $q=1$,
the Koornwinder polynomials admit the  factorization from 
\Cref{cor:Koorn at q=1}:
\begin{equation}\label{eq:Koornwinderfactor}
K_{\lambda}(\z;a,b,c,d;q,t)
	=\prod_{i\ge 1}K_{(1^{\lambda'_i},0^{N-\lambda'_i})}(\z;a,b,c,d;q,t)
	=\prod_{i\ge 1}K_{1^{\lambda'_i}}(\z;a,b,c,d;q,t).
\end{equation}

In this paper we have given  formulas 
(\Cref{thm:main} and \Cref{thm:main2})
in terms of rhombic staircase tableaux 
for the ``rank 1'' Koornwinder polynomials $K_{\lambda}$ 
appearing on the right-hand side of \eqref{eq:Koornwinderfactor}.
Using the factorization  \eqref{eq:Koornwinderfactor}, we also 
gave a combinatorial formula for arbitrary Koornwinder polynomials
at $q=1$ (\Cref{cor:direct}).   
Thus, one might hope to figure out how to insert the parameter $q$ 
into \Cref{cor:direct}, so as to give a formula for general
Koornwinder polynomials.  
In particular, by  analogy with the Macdonald case, 
one might hope that there might be a combinatorial formula 
for arbitrary Koornwinder polynomials where the new combinatorial object  is somehow built 
by stacking  rhombic staircase tableaux on top of each other.  However, as rhombic 
staircase tableaux are already quite complicated, 
and our formulas in \Cref{thm:main} and \Cref{thm:main2} are not that simple,
it is so far unclear to us how to 
define the correct object.

\bibliographystyle{alpha}
\bibliography{bibliography}

\end{document}